\documentclass[10pt,reqno]{article}
\usepackage[a4paper,inner=35mm,outer=35mm,top=35mm,bottom=35mm]{geometry}
\usepackage{fullpage}
\usepackage{times}
\usepackage{mathrsfs}
\usepackage{authblk}
\usepackage{amsfonts,amssymb,amsmath,amsthm}
\usepackage{enumerate}
\usepackage{graphicx,xcolor}
\usepackage{hyperref}
\hypersetup{
    colorlinks=true,
    linkcolor=blue,
	linktocpage=true
}

\newtheorem{theorem}{Theorem}[section]

\newtheorem{lemma}[theorem]{Lemma}

\newtheorem{proposition}[theorem]{Proposition}
\newtheorem{definition}[theorem]{Definition}

\newtheorem{remark}[theorem]{Remark}

\newtheorem{innerAssumption}{Assumption}
\newenvironment{Assumption}[1]
{\renewcommand\theinnerAssumption{#1}\innerAssumption}
{\endinnerAssumption}

\numberwithin{equation}{section}
\newcommand{\cF}{\mathcal{F}}

\newcommand{\cH}{\mathcal{H}}

\newcommand{\cS}{\mathcal{S}}

\def \Ac{{\cal A}}
\def \Bc{{\cal B}}

\def \Fc{{\cal F}}
\def \Gc{{\cal G}}
\def \Hc{{\cal H}}

\def \Pc{{\cal P}}

\def \Wc{{\cal W}}

\def \E{\mathbb{E}}
\def \F{\mathbb{F}}
\def \H{\mathbb{H}}

\def \N{\mathbb{N}}
\def \P{\mathbb{P}}

\def \R{\mathbb{R}}

\def \S{\mathbb{S}}

\def \trans{^{\scriptscriptstyle{\intercal }}}

\def \eps{\varepsilon}

\makeatletter
\makeatletter
\renewcommand*{\@fnsymbol}[1]{\ensuremath{\ifcase#1\or *\or **\or 
\mathsection\or \mathparagraph\or \|\or **\or \dagger\dagger
\or \ddagger\ddagger \else\@ctrerr\fi}}
\makeatother

\makeatletter\def\blfootnote{\xdef\@thefnmark{}\@footnotetext}\makeatother

\begin{document}

\title{A Tikhonov theorem for McKean-Vlasov two-scale systems\\ and a new application to mean field optimal control problems}
\author{Matteo Burzoni\footnote{Universit\`a degli Studi di Milano, Dipartimento di Matematica ``Federigo Enriques'', via Saldini 50, 20133 Milano, Italy; \texttt{matteo.burzoni@unimi.it}}, \; \ Alekos Cecchin\footnote{Universit\`a degli Studi di Padova, Dipartimento di Matematica ``Tullio Levi-Civita'', via Trieste 63, 35121 Padova, Italy; \texttt{alekos.cecchin@unipd.it}}, \; \ Andrea Cosso\footnote{Universit\`a degli Studi di Milano, Dipartimento di Matematica ``Federigo Enriques'', via Saldini 50, 20133 Milano, Italy; \texttt{andrea.cosso@unimi.it}}}

\date{\today}
\maketitle
\abstract{\noindent We provide a new version of the Tikhonov theorem for both two-scale forward systems and also two-scale forward-backward systems of stochastic differential equations, which also covers the McKean-Vlasov case. Differently from what is usually done in the literature, we prove a type of convergence for the ``fast'' variable, which allows the limiting process to be discontinuous. This is relevant for the second part of the paper, where we present a new application of this theory to the approximation of the solution of mean field control problems. Towards this aim, we construct a two-scale system whose ``fast'' component converges to the optimal control process, while the ``slow'' component converges to the optimal state process. The interest in such a procedure is that it allows to approximate the solution of the control problem avoiding the usual step of the minimization of the Hamiltonian.\blfootnote{\textbf{Acknowledgments: }We thank Ulrich Horst and Peter Bank for fruitful discussions. A. Cosso acknowledges support from GNAMPA-INdAM and from the MUR project PRIN 2022 \emph{``Entropy martingale optimal transport and McKean-Vlasov equations''}. In addition, A. Cecchin and A. Cosso acknowledge support from the MUR project PRIN 2022 PNRR \emph{``Probabilistic methods for energy transition''}.}}

\vspace{5mm}


\vspace{5mm}


\section{Introduction}
We present a new version of the stochastic Tikhonov theorem in the framework of McKean-Vlasov stochastic differential equations (SDEs), formulated for two-scale forward systems (Theorem \ref{T:Convergence}) and also for two-scale forward-backward systems (Theorem \ref{T:ConvergenceII}). Besides its own mathematical interest, we are motivated by a novel application of such a result to optimal control problems as we discuss below. We first recall some classical results on Tikhonov theorem and we outline our contribution. In his seminal paper \cite{Tik52}, Tikhonov considered systems of ordinary differential equations (ODEs) on $[0,T]$, of the following form: 
\begin{equation}\label{eq:Tikhonov_ODE}
	\left\{\begin{aligned}
	dx^\varepsilon_t&=b(t,x^\varepsilon_t,a^\varepsilon_t)dt,\\
	\varepsilon da^\varepsilon_t&=B(t,x^\varepsilon_t,a^\varepsilon_t)dt,
\end{aligned}\right.
\end{equation}
with $\varepsilon\in(0,1)$ and initial conditions $x^\varepsilon_0=x_0$, $a^\varepsilon_0=a_0$. Systems of ODEs as in \eqref{eq:Tikhonov_ODE} are sometimes called \emph{two-scale systems}, with ``$x$'' being referred to as the \emph{slow} variable and ``$a$'' as the \emph{fast} variable. By keeping in mind that $\varepsilon$ is supposed to be small, one can heuristically interpret \eqref{eq:Tikhonov_ODE} as follows. For small values of $t$, we expect the slow variable to be close to its initial point $x_0$; on the other hand, by performing the change of variable $s= t/\varepsilon$ in the second equation, we expect the fast variable to be close to its limit point $a_\infty$. Continuing heuristically, sending $\varepsilon$ to zero in the second equation of \eqref{eq:Tikhonov_ODE} yields the equation $B=0$. To make this more rigorous, the idea is to consider the ODE
\begin{equation}\label{eq:Tikhonov_ODE_limit}
	\left\{\begin{aligned}
		dx_t&=b(t,x_t,\widehat{\mathrm a}(t,x_t))dt,\\
		x_0&=x_0,
	\end{aligned}\right.
\end{equation}
together with the condition
\begin{equation}\label{eq:critical_Tikhonov}
B(t,x_t,\widehat{\mathrm a}(t,x_t))=0,\qquad \forall t\in [0,T],
\end{equation} 
where the function $\widehat{\mathrm a}=\widehat{\mathrm a}(t,x)$, satisfies $\widehat{\mathrm a}(0,x_0)=a_\infty$. Under suitable assumptions, the Tikhonov theorem asserts that system \eqref{eq:Tikhonov_ODE} approximates indeed system \eqref{eq:Tikhonov_ODE_limit}--\eqref{eq:critical_Tikhonov}, in the following sense:
	\[
		\lim_{\varepsilon\to0}\sup_{0\le t\le T}|x^\varepsilon_t-x_t|=0,\qquad	\lim_{\varepsilon\to0}\sup_{S\le t\le T}|a^\varepsilon_t-\widehat{\mathrm a}(t,x_t)|=0,
	\]
for every $0<S\le T$.
Kabanov and Pergamenshchikov \cite{book_two_scales} studied a stochastic version of the above theorem by introducing the stochastic counterpart of \eqref{eq:Tikhonov_ODE}:
\begin{equation}\label{eq:Kabanov_SDE}
	\left\{\begin{aligned}
		dX^\varepsilon_t&=b(t,X^\varepsilon_t,A^\varepsilon_t)dt+\sigma(t,X^\varepsilon_t,A^\varepsilon_t)dW_t^x,\\
		\varepsilon dA^\varepsilon_t&=B(t,X^\varepsilon_t,A^\varepsilon_t)dt+\beta^\varepsilon\Sigma(t,X^\varepsilon_t,A^\varepsilon_t)dW_t^a,
	\end{aligned}\right.
\end{equation}
together with the initial conditions $X^\varepsilon_0=x_0~$, $A^\varepsilon_0=a_0$, and $\beta^\varepsilon=o(\varepsilon)$. Here $W^x$ and $W^a$ are two independent Brownian motions (see Remark \ref{R:TwoW}). The stochastic counterpart of \eqref{eq:Tikhonov_ODE_limit} reads as
\begin{equation}\label{eq:Kabanovv_SDE_limit}
	\left\{\begin{aligned}
		dX_t&=b(t,X_t,\widehat\alpha(t,X_t))dt+\sigma(t,X_t,\widehat\alpha(t,X_t))dW_t^x,\\
		X_0&=x_0,
	\end{aligned}\right.
\end{equation}
with $\widehat\alpha$ satisfying \eqref{eq:critical_Tikhonov}. They prove, under suitable conditions, that 
\[
	P-\lim_{\varepsilon\to0}\sup_{0\le t\le T}|X^\varepsilon_t-X_t|=0,\qquad
	P-\lim_{\varepsilon\to0}\sup_{S\le t\le T}|A^\varepsilon_t-\widehat\alpha(t,X_t)|=0,
\]
where the limits are understood in the sense of convergence in probability. The proof consists of two steps. First, the result is established on a short time interval, using the continuous dependence of the solution of a system of SDEs on the parameter $\varepsilon$. Second, a stability result is proven in the sense that if the solutions of \eqref{eq:Kabanov_SDE} and \eqref{eq:Kabanovv_SDE_limit} are close on a short time interval, they remain close on $[0,T]$. We refer to the monograph \cite{book_two_scales} for a thorough analysis of two-scale stochastic systems. Along the same lines, the result has been generalized to the infinite-dimensional case in \cite{BG06}, see also \cite{Schwiech21}. 
A different strand of literature, which has its origin in Bogoliubov's averaging principle, studies the case of $\beta^\varepsilon=\sqrt{\varepsilon}$. Unlike the present setting, the formal limit as $\varepsilon\to 0$ in the second equation of \eqref{eq:Kabanov_SDE} does not correspond to $B=0$, but to an ergodic SDE for the fast variable. This case requires completely different techniques and we refer to \cite{CF09,book_Random_pert} for classical references, as well as to \cite{bandini2024singular,GT21,GT22,RSX21} for more recent results.

Our first main results are Theorems \ref{T:Convergence} and \ref{T:ConvergenceII}, which correspond to our two formulations of the Tikhonov theorem. In particular, Theorem \ref{T:Convergence} applies to the McKean-Vlasov version of \eqref{eq:Kabanov_SDE} (see system \ref{TwoScale}). On the other hand, Theorem \ref{T:ConvergenceII} applies to a class of coupled forward-backward systems which arise in the study of mean field control problems, as described in Section \ref{sec:control}. In both theorems we prove a stronger convergence on the slow variable and a weaker convergence on the fast variable. The latter allows to get a possibly discontinuous process in the limit, which is relevant for our application to optimal control problems. The second novelty is that our proof is based on a completely different approach with respect to the classical literature and this allows us to easily include the McKean-Vlasov case. We stress however that this result is new even in the classical setup.

In the second part of the paper we use two-scale systems to approximate stochastic optimal control problems. In particular, as explained below, we rely on the Pontryagin stochastic maximum principle, even though it is worth mentioning that another approach could be investigated relying on Bellman's optimality principle and on the representation of the value function in terms of a suitable forward-backward system of stochastic differential equations (see for instance \cite{PardouxPeng92,Bandini18,Bandini19,BCC,BCFP18,BandiniFuhrman17}).

It is well known that the Pontryagin stochastic maximum principle reduces the problem to solving a system of forward-backward stochastic differential equations (FBSDEs) of the following form:
\begin{equation}\label{eq:general_FBSDE}
	\left\{\begin{aligned}
		dX_t&=b(t,X_t,\P_{X_t},\widehat{\alpha}(t,\Theta_t,\P_{\Theta_t}),\P_{\widehat{\alpha}_t})dt+\sigma(t,X_t,\P_{X_t},\widehat{\alpha}(t,\Theta_t,\P_{\Theta_t}),\P_{\widehat{\alpha}_t})dW_t^x,\\
		dY_t&=F(t,X_t,\P_{X_t},\widehat{\alpha}(t,\Theta_t,\P_{\Theta_t}),\P_{\widehat{\alpha}_t},Y_t, Z_t)dt+Z_tdW_t^x,\\
		X_0&=\xi,\ Y_T=G(X_T,\P_{X_T}),
	\end{aligned}\right.
\end{equation}
where $\Theta_t:=(X_t,Y_t,Z_t)$ and $\widehat{\alpha}$, $F$, $G$ are suitable functions. A crucial role is played by the function $\widehat{\alpha}$, representing the minimizer of the Hamiltonian of the system. In very simple cases such a function can be explicitly calculated and the solution of \eqref{eq:general_FBSDE} yields the solution to the control problem. However, even if it is explicit, when the dimension of the problem is large, calculating $\widehat\alpha$ at each time step may be computationally quite expensive (as a matter of fact, it may require solving a large system of linear equations, as it happens in the linear quadratic problem). Similarly, when $\widehat{\alpha}$ is not explicit, it needs to be approximated numerically and this minimization step becomes a non-trivial part of the algorithm for solving \eqref{eq:general_FBSDE} (see, e.g., the introduction of \cite{hu_MFlangevin} for a discussion).

In our last main result (Theorem \ref{thm:main_SMP}) we prove that the solution of a suitable two-scale system of FBSDEs, see \eqref{eq:two scales_FBSDE}, converges to the solution of \eqref{eq:general_FBSDE}. In particular, the ``fast'' component $A^\eps$ of \eqref{eq:two scales_FBSDE} converges towards the optimal control process given by $\widehat{\alpha}(t,\Theta_t,\P_{\Theta_t})$, while the ``slow'' component $X^\eps$ converges towards the optimal state process. The advantage is that such a procedure avoids the usual minimization step. As a matter of fact, the problem of finding the minimizer of the Hamiltonian is delegated to the equation for $A^\eps$, which runs on a fast time scale. Following the above heuristic discussion on the Tikhonov theorem, the equation for $A^\eps$ is chosen in such a way that in the limit equation we end up with the equation $\partial_a H=0$, which corresponds to finding the critical points of the Hamiltonian (see also Remark \ref{R:Form_of_the_FBSDE}). A (strict) convexity assumption, typical of the stochastic maximum principle, ensures that critical points are unique and they give indeed the unique minimizer of $H$ (and in turn the optimal control). We point out that we also consider the so-called extended mean field control setting, namely we allow $H$ to depend on the law of the control itself, therefore the equation for the critical points is more complicated, see \eqref{eq:critical}. Finally, we emphasize that the approximation procedure is new even for classical (non mean field) control problems.
 
We end the paper by illustrating our method in the case of McKean-Vlasov linear quadratic control problems, considering two different cases. Firstly, we study the one-dimensional case, where we can easily solve the FBSDE system \eqref{eq:general_FBSDE} and compare it with \eqref{eq:two scales_FBSDE} for different values of $\eps$, given by the two-scale approximation induced by the Tikhonov theorem. Secondly, we consider the (non McKean-Vlasov) linear quadratic problem in large dimension, for which we show that, even in such a simple framework, the classical approach is severely outperformed by the two-scale approximation (see Table \ref{table:dimensions}). The latter example is relevant because, as already emphasized, the results of the paper are new even in the non McKean-Vlasov setting. Those results can be explained by the fact that in the classical algorithm the minimization of the Hamiltonian requires \emph{solving}, at each time step, a large system of linear equations; on the other hand, using our method just amounts at \emph{evaluating} at each time step such a linear term, which indeed corresponds to the drift in the dynamics of the fast variable. 

The rest of the paper is organized as follows. In Section \ref{sec:Tik} we prove two different formulations of the stochastic Tikhonov theorem: Theorem \ref{T:Convergence}, where it is studied the two-scale stochastic system \eqref{TwoScale} and its convergence towards the limiting equations \eqref{widehat_alpha}-\eqref{widehat_X}; Theorem \ref{T:ConvergenceII}, where the two-scale forward-backward system \eqref{TwoScaleFB} is studied and the limiting equations are \eqref{widehat_alpha_FB}-\eqref{widehat_X_FB}. In Section \ref{sec:control} we introduce the mean field optimal control problem \eqref{eq:state}-\eqref{eq:cost}, on which we impose a suitable set of assumptions guaranteeing the necessary and sufficiency parts of Pontryagin maximum principle. Then, we consider the McKean-Vlasov forward-backward system \eqref{eq:SMP_FBSDE} arising from the Pontryagin maximum principle, which, together with \eqref{eq:critical}, turns out to be the limiting equation of the two-scale stochastic system \eqref{eq:two scales_FBSDE}. The last part of Section \ref{sec:control} is devoted to develop a numerical example in the linear quadratic case. Finally, in Appendix \ref{S:App} we recall the definition of Lions differentiability and present its essential features.

\section{Two-scale stochastic systems: stochastic Tikhonov theorems}\label{sec:Tik}

Let $(\Omega,\cF,\P)$ be a complete probability space, on which a $m$-dimensional Brownian motion $W=(W_t)_{t\geq0}$ is defined. For every $n\in\N$ and any sub-$\sigma$-algebra $\Hc$ of $\Fc$, let $L^2(\Omega,\Hc,\P;\R^n)$, or simply $L^2$ when no confusion arises, the space of (equivalence classes of) $\Hc$-measurable random variables taking values in $\R^n$. Moreover, let $\Pc(\R^n)$ be the family of all probability measures on $(\R^n,\Bc(\R^n))$, where $\Bc(\R^n)$ is the Borel $\sigma$-algebra on $\R^n$. Given a random variable $\xi\colon\Omega\rightarrow\R^n$ we denote by $\P_\xi$ its law on $(\R^n,\Bc(\R^n))$.\\
We define
\[
\Pc_2(\R^n) \ := \ \left\{\mu \in \Pc(\R^n)\colon \int_{\R^n} |x|^2\,\mu(dx)<+\infty\right\}.
\]
On the set $\Pc_2(\R^n)$ we consider the $2$-Wasserstein metric
\begin{equation*}
	\Wc_2(\mu,\mu') := \inf\bigg\{\bigg(\int_{\R^n\times\R^n}|x-y|^2\,\pi(dx,dy)\bigg)^{1/2}\colon
	\pi \in \Pc(\R^n\times\R^n) \text{ such that } \;\pi_x= \mu,\ \pi_y=\mu'\bigg\},
\end{equation*}
where $\pi_x$, $\pi_y$ are the respective marginals. Given $\xi,\xi'\in L^2(\Omega,\Fc,\P;\R^n)$ with laws $\mu$ and $\mu'$, respectively, it follows easily from the definition of $\Wc_2$ that
\begin{equation}\label{EstimateW2}
\Wc_2(\mu,\mu') \ \leq \ \|\xi-\xi'\|_{L^2},
\end{equation}
where $\|\cdot\|_{L^2}$ denotes the $L^2$-norm. We also denote
\[
\|\mu\|_2 \ := \ \sqrt{\Wc_2(\mu,\delta_0)} \ = \ \bigg(\int_\R|x|^2\,\mu(dx)\bigg)^{1/2}, \qquad \forall\,\mu\in\Pc_2(\R^n),
\]
where $\delta_0$ is the Dirac delta centered at $0$. Notice that if $\xi\colon\Omega\rightarrow\R^n$ is a random variable having distribution $\mu$, then
\begin{equation}\label{NormWass}
\|\mu\|_2 \ = \ \|\xi\|_{L^2}.
\end{equation}
In the sequel we also consider the set $\Pc_2(\R^{n\times k})$ and the corresponding Wasserstein metric $\Wc_2$, which are defined in an analogous manner.\\ 
Let $\F^W=(\cF_t^W)_{t\geq0}$ denote the $\P$-completion of the filtration generated by the Brownian motion $W$. Suppose also that there exists a sub-$\sigma$-algebra $\Gc$ of $\Fc$ satisfying:
\begin{enumerate}[1)]
\item $\Gc$ and $\Fc_\infty^W$ are independent;
\item there exists a $\Gc$-measurable random variable $U\colon\Omega\rightarrow\R$ having uniform distribution on $[0,1]$.
\end{enumerate}

\begin{remark}\label{R:Uniform}
We refer to \cite[Lemma 2.1, Remarks 2.2 and 2.3]{CKGPR} for some comments and insights on the $\sigma$-algebra $\Gc$. Here we just notice that when we deal with McKean-Vlasov stochastic differential equations it is natural to consider \emph{random} initial conditions. Such random initial conditions will be taken $\Gc$-measurable. The second property satisfied by $\Gc$ is equivalent to the following property, see \cite[Lemma 2.1]{CKGPR}:
\[
\textup{$\Gc$ satisfies 2)} \ \ \ \Longleftrightarrow \ \ \ \forall\,n\in\N,\;\text{$\Pc_2(\R^n)=\big\{\P_\xi\colon\xi\in L^2(\Omega,\Gc,\P;\R^n)\big\}$}.
\]
\end{remark}

\noindent Let $\F=(\cF_t)_{t\geq0}$ denote the the filtration given by
\begin{equation}\label{Filtr}
\Fc_t \ = \ \Gc\vee\Fc_t^W, \qquad \forall\,t\geq0.
\end{equation}

\noindent Given $n,\tilde n\in\N$, we introduce the following spaces of $\R^n$-valued or $\R^{n\times\tilde n}$-valued stochastic processes on $[0,T]$.
\begin{itemize}
\item $\S_n^2$ is the set of continuous and $\F$-adapted $\R^n$-valued stochastic processes $Y=(Y_t)_{t\in[0,T]}$ such that
\[
\|Y\|_{\S^2}^2 \ := \ \E\Big[\sup_{0\leq t\leq T}|Y_t|^2\Big] \ < \ \infty.
\]
\item $\H_n^2$ (resp. $\H_{n\times\tilde n}^2$) is the set of $\F$-progressively measurable $\R^n$-valued (resp. $\R^{n\times\tilde n}$-valued) stochastic processes $Z=(Z_t)_{t\in[0,T]}$ such that
\[
\|Z\|_{\H^2}^2 \ := \ \E\bigg[\int_0^T|Z_t|^2dt\bigg] \ < \ \infty.
\]
\end{itemize}

\subsection{Stochastic Tikhonov theorem I}
\label{subsec:TikI}

Given $T>0$, $d,k\in\N$, $\xi\in L^2(\Omega,\Gc,\P;\R^d)$, $\eta\in L^2(\Omega,\Gc,\P;\R^k)$, we consider, for every $\varepsilon>0$, the following two-scale system of stochastic differential equations on $[0,T]$:
\begin{equation}\label{TwoScale}
	\left\{\begin{aligned}dX^\varepsilon_t&=b\big(t,X^\varepsilon_t,\P_{X_t^\varepsilon},A_t^\varepsilon,\P_{A_t^\varepsilon}\big)dt+\sigma\big(t,X^\varepsilon_t,\P_{X_t^\varepsilon},A_t^\varepsilon,\P_{A_t^\varepsilon}\big) dW_t,\\
		\varepsilon dA^\varepsilon_t&=B\big(t,X^\varepsilon_t,\P_{X_t^\varepsilon},A_t^\varepsilon,\P_{A_t^\varepsilon}\big)dt+\beta^\varepsilon \Sigma\big(t,X^\varepsilon_t,\P_{X_t^\varepsilon},A_t^\varepsilon,\P_{A_t^\varepsilon}\big)dW_t,\\
		X^\varepsilon_0&=\xi,\ A^\varepsilon_0=\eta,
	\end{aligned}\right.
\end{equation}
where we impose the following assumptions on $\beta^\eps$ and on the coefficients
\[
b\,,\,\sigma\,,\,B\,,\,\Sigma \ \colon[0,T]\times\Omega\times\R^d\times\Pc_2(\R^d)\times\R^k\times\Pc_2(\R^k) \ \ \ \longrightarrow \ \ \ \R^d\,,\,\R^{d\times m}\,,\,\R^k\,,\,\R^{k\times m}.
\]
\begin{Assumption}{\bf(A)}\label{ass:two_scale}\quad
\begin{enumerate}[\upshape1)]
\item $\beta^\eps=o(\eps)$ as $\eps\rightarrow0$, that is $\lim_{\varepsilon\rightarrow0}\frac{\beta^\varepsilon}{\varepsilon}=0$.
\item The functions $b$, $\sigma$, $B$, $\Sigma$ are measurable with respect to $Prog\otimes\Bc(\R^d)\otimes\Bc(\Pc_2(\R^d))\otimes\Bc(\R^k)\otimes\Bc(\Pc_2(\R^k))$, where $Prog$ denotes the $\sigma$-algebra of $\F$-progressive sets on $[0,T]\times\Omega$, while $\Bc(S)$ denotes the Borel $\sigma$-algebra of a topological space $S$.
\item $b$, $\sigma$, $B$, $\Sigma$ satisfy linear growth conditions: there exists a constant $K>0$ such that
\begin{align*}
|b(t,\omega,x,\mu,a,\nu)| + |\sigma(t,\omega,x,\mu,a,\nu)| \ &\leq \ K\big(1 + |x| + \|\mu\|_2 + |a| + \|\nu\|_2\big), \\
|B(t,\omega,x,\mu,a,\nu)| + |\Sigma(t,\omega,x,\mu,a,\nu)| \ &\leq \ K\big(1 + |x| + \|\mu\|_2 + |a| + \|\nu\|_2\big),
\end{align*}
for all $(t,\omega,x,\mu,a,\nu)\in[0,T]\times\Omega\times\R^d\times\Pc_2(\R^d)\times\R^k\times\Pc(\R^k)$.
\item $b$, $\sigma$, $B$, $\Sigma$ satisfy the Lipschitz continuity condition: there exists a constant $K>0$ such that
\begin{align*}
|b(t,\omega,x,\mu,a,\nu) - b(t,\omega,x',\mu',a',\nu')| &\leq \ K\big(|x-x'|+\Wc_2(\mu,\mu')+|a-a'|+\Wc_2(\nu,\nu')\big), \\
|\sigma(t,\omega,x,\mu,a,\nu) - \sigma(t,\omega,x',\mu',a',\nu')| &\leq \ K\big(|x-x'|+\Wc_2(\mu,\mu')+|a-a'|+\Wc_2(\nu,\nu')\big), \\
|B(t,\omega,x,\mu,a,\nu) - B(t,\omega,x',\mu',a,\nu)| &\leq \ K\big(|x-x'|+\Wc_2(\mu,\mu')\big), \\
|\Sigma(t,\omega,x,\mu,a,\nu) - \Sigma(t,\omega,x',\mu',a',\nu')| &\leq \ K\big(|x-x'|+\Wc_2(\mu,\mu')+|a-a'|+\Wc_2(\nu,\nu')\big),
\end{align*}
for all $(t,\omega)\in[0,T]\times\Omega$, $(x,\mu,a,\nu),(x',\mu',a',\nu')\in\R^d\times\Pc_2(\R^d)\times\R^k\times\Pc_2(\R^k)$.
\item $B$ satisfies the monotonicity condition: there exists a constant $\lambda>0$ such that
\begin{equation}\label{Monotonicity}
\E\Big[\big\langle B\big(t,\xi,\P_\xi,\eta,\P_\eta\big) - B\big(t,\xi,\P_\xi,\eta',\P_{\eta'}\big),\eta - \eta'\big\rangle\Big] \ \leq \ - \lambda\E\big[|\eta - \eta'|^2\big],
\end{equation}
for all $t\in[0,T]$, $\xi\in L^2(\Omega,\Fc,\P;\R^d)$, $\eta,\eta'\in L^2(\Omega,\Fc,\P;\R^k)$.
\end{enumerate}
\end{Assumption}

\begin{remark}
As in \eqref{Monotonicity}, $B(t,\omega,x,\mu,a,\nu)$ is often written omitting $\omega$, that is $B(t,x,\mu,a,\nu)$. Similarly for the other coefficients $b$, $\sigma$, $\Sigma$.
\end{remark}

\begin{remark}\label{R:TwoW}
Notice that in \cite{book_two_scales} the Brownian motions driving the equations of the two-scale system are different and independent. Here we relax this assumption taking in \eqref{TwoScale} the same $m$-dimensional Brownian motion $W$. More precisely, the case studied in \cite{book_two_scales} is obtained writing $W=(W^x,W^a)$ (with $W^x$ being $m_x$-dimensional and $W^a$ being $m_a$-dimensional, where $m=m_x+m_a$) and choosing the coefficients $\sigma$, $\Sigma$ in such a way that the first equation in \eqref{TwoScale} only involves $W^x$, while the second equation only involves $W^a$.
\end{remark}

\begin{lemma}\label{L:Estimate}
Let Assumption \ref{ass:two_scale} hold. Let $\xi\in L^2(\Omega,\Gc,\P;\R^d)$, $\alpha,\alpha'\in\H_k^2$. Suppose that $X^\alpha\in\S_d^2$ is a solution to the following controlled stochastic differential equation on $[0,T]$:
\begin{equation}\label{SDE_X_alpha_Estimate}
dX_t \ = \ b\big(t,X_t,\P_{X_t},\alpha_t,\P_{\alpha_t}\big)dt+\sigma\big(t,X_t,\P_{X_t},\alpha_t,\P_{\alpha_t}\big)dW_t, \qquad \forall\,t\in[0,T], \qquad\quad X_0 \ = \ \xi.
\end{equation}
Similarly, suppose that $X^{\alpha'}\in\S_d^2$ is a solution to equation \eqref{SDE_X_alpha_Estimate} with $\alpha'$ in place of $\alpha$. Then, for any $\underline t,\bar t\in[0,T]$, with $\underline t\leq\bar t$, it holds that
\begin{align}\label{EstimateX-X'}
\E\Big[\sup_{\underline t\leq t\leq\bar t}\big|X_t^\alpha - X_t^{\alpha'}\big|^2\Big] \ &\leq \ \textup{e}^{C_K(1+T^2)}\bigg\{\E\big[\big|X_{\underline t}^\alpha - X_{\underline t}^{\alpha'}\big|^2\big] \\
&\quad \ + \E\bigg[\int_{\underline t}^{\bar t}\big|b\big(s,X_s^{\alpha'},\P_{X_s^{\alpha'}},\alpha_s,\P_{\alpha_s}\big) - b\big(s,X_s^{\alpha'},\P_{X_s^{\alpha'}},\alpha_s',\P_{\alpha_s'}\big)\big|^2 ds\bigg] \notag \\
&\quad \ + \E\bigg[\int_{\underline t}^{\bar t}\big|\sigma\big(s,X_s^{\alpha'},\P_{X_s^{\alpha'}},\alpha_s,\P_{\alpha_s}\big) - \sigma\big(s,X_s^{\alpha'},\P_{X_s^{\alpha'}},\alpha_s',\P_{\alpha_s'}\big)\big|^2 ds\bigg]\bigg\}, \notag
\end{align}
for some constant $C_K\geq0$, depending only on the constant $K$ appearing in Assumption \ref{ass:two_scale}. In particular, it holds that
\begin{equation}\label{EstimateX-X'2}
\E\Big[\sup_{0\leq t\leq\bar t}\big|X_t^\alpha - X_t^{\alpha'}\big|^2\Big] \ \leq \ 8K^2 \textup{e}^{C_K(1+T^2)}\E\bigg[\int_0^{\bar t} \big|\alpha_s - \alpha_s'\big|^2 ds\bigg].
\end{equation}
\end{lemma}
\begin{proof}
Fix $\underline t,\bar t\in[0,T]$, with $\underline t\leq\bar t$. We begin noting that estimate \eqref{EstimateX-X'2} follows directly from \eqref{EstimateX-X'}, the Lipschitz property of $b,\sigma$ with respect to their two last arguments $(a,\nu)$, and estimate \eqref{EstimateW2}. Therefore it remains to prove \eqref{EstimateX-X'}. From equation \eqref{SDE_X_alpha} we have, for any $t\in[\underline t,\bar t]$,
\begin{align*}
\big|X_t^\alpha - X_t^{\alpha'}\big|^2 \ &\leq \ 3\big|X_{\underline t}^\alpha - X_{\underline t}^{\alpha'}\big|^2 + 3\bigg|\int_{\underline t}^t \Big( b\big(s,X_s^\alpha,\P_{X_s^\alpha},\alpha_s,\P_{\alpha_s}\big) -  b\big(s,X_s^{\alpha'},\P_{X_s^{\alpha'}},\alpha_s',\P_{\alpha_s'}\big)\Big) ds\bigg|^2 \\
&\quad \ + 3\bigg|\int_{\underline t}^t \Big(\sigma\big(s, X_s^\alpha,\P_{ X_s^\alpha},\alpha_s,\P_{\alpha_s}\big) - \sigma\big(s, X_s^{\alpha'},\P_{ X_s^{\alpha'}},\alpha_s',\P_{\alpha_s'}\big)\Big)dW_s\bigg|^2.
\end{align*}
Taking the supremum over $t\in[\underline t,r]$, with $r\in[\underline t,\bar t]$, together with the expectation, and also applying the Jensen and Burkholder-Davis-Gundy inequalities, we obtain (for some constant $C_2\geq0$)
\begin{align*}
&\E\Big[\sup_{\underline t\leq t\leq r}\big| X_t^\alpha -  X_t^{\alpha'}\big|^2\Big] \\
&\leq \ 3\E\Big[\big|X_{\underline t}^\alpha - X_{\underline t}^{\alpha'}\big|^2\Big] + 3\E\bigg[\bigg(\int_{\underline t}^r \Big|b\big(s, X_s^\alpha,\P_{ X_s^\alpha},\alpha_s,\P_{\alpha_s}\big) -  b\big(s, X_s^{\alpha'},\P_{ X_s^{\alpha'}},\alpha_s',\P_{\alpha_s'}\big)\Big| ds\bigg)^2\bigg] \\
&+ 3\E\bigg[\sup_{\underline t\leq t\leq r}\bigg|\int_{\underline t}^t \Big(\sigma\big(s, X_s^\alpha,\P_{ X_s^\alpha},\alpha_s,\P_{\alpha_s}\big) - \sigma\big(s, X_s^{\alpha'},\P_{ X_s^{\alpha'}},\alpha_s',\P_{\alpha_s'}\big)\Big)dW_s\bigg|^2\bigg] \\
&\leq \ 3\E\Big[\big|X_{\underline t}^\alpha - X_{\underline t}^{\alpha'}\big|^2\Big] + 3T\E\bigg[\int_{\underline t}^r \Big| b\big(s, X_s^\alpha,\P_{ X_s^\alpha},\alpha_s,\P_{\alpha_s}\big) -  b\big(s, X_s^{\alpha'},\P_{ X_s^{\alpha'}},\alpha_s',\P_{\alpha_s'}\big)\Big|^2 ds\bigg] \\
&+ 3C_2 \E\bigg[\int_{\underline t}^r \Big|\sigma\big(s, X_s^\alpha,\P_{ X_s^\alpha},\alpha_s,\P_{\alpha_s}\big) - \sigma\big(s, X_s^{\alpha'},\P_{ X_s^{\alpha'}},\alpha_s',\P_{\alpha_s'}\big)\Big|^2 ds\bigg].
\end{align*}
By the Lipschitz property of $b$ and $\sigma$, we find
\begin{align*}
\E\Big[\sup_{\underline t\leq t\leq r}\big|X_t^\alpha - X_t^{\alpha'}\big|^2\Big] \ &\leq \ 3\E\Big[\big|X_{\underline t}^\alpha - X_{\underline t}^{\alpha'}\big|^2\Big] + 24K^2(T + C_2)\int_{\underline t}^r \E\Big[\sup_{\underline t\leq t\leq s}\big| X_t^\alpha -  X_t^{\alpha'}\big|^2\Big] ds \\
&+ 6(T+C_2)\E\bigg[\int_{\underline t}^r \Big| b\big(s, X_s^{\alpha'},\P_{ X_s^{\alpha'}},\alpha_s,\P_{\alpha_s}\big) -  b\big(s, X_s^{\alpha'},\P_{ X_s^{\alpha'}},\alpha_s',\P_{\alpha_s'}\big)\Big|^2 ds \\
&+ \int_{\underline t}^r \Big|\sigma\big(s, X_s^{\alpha'},\P_{ X_s^{\alpha'}},\alpha_s,\P_{\alpha_s}\big) - \sigma\big(s, X_s^{\alpha'},\P_{ X_s^{\alpha'}},\alpha_s',\P_{\alpha_s'}\big)\Big|^2 ds\bigg].
\end{align*}
Applying Gronwall's inequality to the function $v(r):=\E\sup_{\underline t\leq t\leq r}|X_t^\alpha - X_t^{\alpha'}|^2$, $r\in[\underline t,\bar t]$, we obtain
\begin{align*}
&\E\Big[\sup_{\underline t\leq t\leq\bar t}\big| X_t^\alpha -  X_t^{\alpha'}\big|^2\Big] \ \leq \ 3\text{e}^{24K^2(T+C_2)T}\E\Big[\big|X_{\underline t}^\alpha - X_{\underline t}^{\alpha'}\big|^2\Big] \notag \\
&+ 6(T+C_2)\text{e}^{24K^2(T+C_2)T}\E\bigg[\int_{\underline t}^{\bar t} \Big| b\big(s, X_s^\alpha,\P_{ X_s^\alpha},\alpha_s,\P_{\alpha_s}\big) -  b\big(s, X_s^{\alpha'},\P_{ X_s^{\alpha'}},\alpha_s',\P_{\alpha_s'}\big)\Big|^2 ds \\
&+ \int_{\underline t}^{\bar t} \Big|\sigma\big(s, X_s^\alpha,\P_{ X_s^\alpha},\alpha_s,\P_{\alpha_s}\big) - \sigma\big(s, X_s^{\alpha'},\P_{ X_s^{\alpha'}},\alpha_s',\P_{\alpha_s'}\big)\Big|^2 ds\bigg]. \notag
\end{align*}
This proves \eqref{EstimateX-X'} for some constant $C_K\geq0$, depending only on $K$ and satisfying $(3+6(T+C_2))\exp(24K^2(T+C_2)T)\leq\exp(C_K(1+T^2))$.
\end{proof}

\begin{proposition}\label{P:System_eps}
Let Assumption \ref{ass:two_scale} hold. For every $\varepsilon>0$, there exists a unique solution $(X^\varepsilon,A^\varepsilon)$ in $\S_d^2\times\S_k^2$ to system \eqref{TwoScale}. In addition, there exists a constant $C>0$, depending only on $T,K,\lambda$, but independent of $\varepsilon$, such that
\begin{equation}\label{EstimateA}
\|A^\varepsilon\|_{\H^2}^2 \ \leq \ C\bigg(1+\frac{(\beta^\eps)^2}{\eps^2}\bigg)\big(1 + \E|\eta|^2 + \E|\xi|^2\big)\textup{e}^{C\left(1+\frac{(\beta^\eps)^2}{\eps^2}\right)}
\end{equation}
and
\begin{equation}\label{EstimateX}
\|X^\eps\|_{\S^2}^2 \ \leq \ C\bigg(1+\frac{(\beta^\eps)^2}{\eps^2}\bigg)\big(1 + \E|\eta|^2 + \E|\xi|^2\big)\textup{e}^{C\left(1+\frac{(\beta^\eps)^2}{\eps^2}\right)}.
\end{equation}
\end{proposition}
\begin{proof}
\emph{Existence and uniqueness of $(X^\varepsilon,A^\varepsilon)$.} Let $\Phi\colon\H_k^2\rightarrow\H_k^2$ be the map defined as follows. Given $\alpha\in\H_k^2$, let $X^\alpha\in\S_d^2$ be the solution to the following controlled stochastic differential equation on $[0,T]$:
\begin{equation}\label{SDE_X_alpha}
dX_t \ = \ b\big(t,X_t,\P_{X_t},\alpha_t,\P_{\alpha_t}\big)dt + \sigma\big(t,X_t,\P_{X_t},\alpha_t,\P_{\alpha_t}\big)dW_t, \qquad X_0 \ = \ \xi.
\end{equation}
The existence and uniqueness of $X^\alpha\in\S_d^2$ follows from the Lipschitz and linear growth conditions of $b$ and $\sigma$ (see for instance \cite[Theorem 4.21]{bookMFG}). Then, let $\Phi(\alpha)\in\H_k^2$ be the unique solution to the following stochastic differential equation on $[0,T]$:
\[
\eps dA_t \ = \ B\big(t,X_t^\alpha,\P_{X_t^\alpha},A_t,\P_{A_t}\big)dt + \beta^\eps\Sigma\big(t,X_t^\alpha,\P_{X_t^\alpha},A_t,\P_{A_t}\big)dW_t, \qquad A_0 \ = \ \eta.
\]
The existence and uniqueness of $\Phi(\alpha)\in\H_k^2$ can be proved proceeding as in \cite[Theorem 3.21]{PardouxRascanu}, and it follows from the monotonicity condition of $B$ and the Lipschitz property of $\Sigma$ with respect to $(a,\nu)$. This shows that $\Phi$ is well-defined. It remains to prove that $\Phi$ is a contraction, indeed, the fixed point of $\Phi$ is clearly the desired solution $A^\varepsilon$. By the linear growth conditions of $B$ and $\Sigma$ with respect to $(a,\nu)$ (Assumption \ref{ass:two_scale}-4)), it is easy to see that $A^\eps$ also belongs to $\S_d^2$. Moreover, taking $\alpha:=A^\varepsilon$ in \eqref{SDE_X_alpha} we see that $X^{A^\varepsilon}$ corresponds to $X^\varepsilon$.

Let us prove that $\Phi$ is a contraction. Let $\alpha,\alpha'\in\H_k^2$. Without loss of generality, we can suppose that $\alpha_0=\alpha_0'$. If this is not the case, we define $\tilde\alpha\in\H_k^2$ as
\begin{equation}\label{tilde_alpha}
\tilde\alpha_t \ := \
\begin{cases}
\alpha_0', \qquad &t=0, \\
\alpha_t, \qquad &t\in(0,T],
\end{cases}
\end{equation}
and replace $\alpha$ by $\tilde\alpha$.
From It\^o's formula applied to $|\Phi(\alpha)_t-\Phi(\alpha')_t|^2$, we get
\begin{align*}
&\big|\Phi(\alpha)_t - \Phi(\alpha')_t\big|^2 \ = \\
&= \ \frac{2}{\varepsilon}\int_0^t \big\langle B\big(s,X_s^\alpha,\P_{X_s^\alpha},\Phi(\alpha)_s,\P_{\Phi(\alpha)_s}\big) - B\big(s,X_s^{\alpha'},\P_{X_s^{\alpha'}},\Phi(\alpha')_s,\P_{\Phi(\alpha')_s}\big),\Phi(\alpha)_s - \Phi(\alpha')_s\big\rangle ds \\
&+ \frac{(\beta^\eps)^2}{\eps^2}\int_0^t \big|\Sigma\big(s,X_s^\alpha,\P_{X_s^\alpha},\Phi(\alpha)_s,\P_{\Phi(\alpha)_s}\big) - \Sigma\big(s,X_s^{\alpha'},\P_{X_s^{\alpha'}},\Phi(\alpha')_s,\P_{\Phi(\alpha')_s}\big)\big|^2ds \\
&+ \frac{2\beta^\varepsilon}{\varepsilon}\!\!\!\int_0^t \!\!\!\!\big\langle\big(\Sigma\big(s,X_s^\alpha,\P_{X_s^\alpha},\Phi(\alpha)_s,\P_{\Phi(\alpha)_s}\big) - \Sigma\big(s,X_s^{\alpha'},\P_{X_s^{\alpha'}},\Phi(\alpha')_s,\P_{\Phi(\alpha')_s}\big)\big)\trans(\Phi(\alpha)_s - \Phi(\alpha')_s),dW_s\big\rangle.
\end{align*}
Notice that the last term above is a true martingale (with zero mean) as $X^\alpha,X^{\alpha'}\in\S_d^2$, $\Phi(\alpha),\Phi(\alpha')\in\H_k^2$, and $\Sigma$ satisfies a linear growth condition (see item 2 of Assumption \ref{ass:two_scale}). As a consequence, taking the expectation, using the monotonicity condition \eqref{Monotonicity}, and also the Lipschitz property of $B$ in $(x,\mu)$ and of $\Sigma$ in $(x,\mu,a,\nu)$, we obtain
\begin{align*}
&\E\big|\Phi(\alpha)_t - \Phi(\alpha')_t\big|^2 \ \leq \ - \frac{2\lambda}{\varepsilon}\E\int_0^t \big|\Phi(\alpha)_s - \Phi(\alpha')_s\big|^2 ds \\
&+ \frac{2K}{\varepsilon}\E\int_0^t \Big(\big| X_s^\alpha - X_s^{\alpha'}\big|+\Wc_2\big(\P_{X_s^\alpha},\P_{X_s^{\alpha'}}\big)\Big)\big|\Phi(\alpha)_s - \Phi(\alpha')_s\big| ds \\
&+ \frac{4K^2(\beta^\varepsilon)^2}{\varepsilon^2}\E\int_0^t \Big(\big| X_s^\alpha - X_s^{\alpha'}\big|^2+\Wc_2\big(\P_{X_s^\alpha},\P_{X_s^{\alpha'}}\big)^2 \\
&+ \big|\Phi(\alpha)_s - \Phi(\alpha')_s\big|^2 + \Wc_2\big(\P_{\Phi(\alpha)_s},\P_{\Phi(\alpha')_s}\big)^2\Big)ds \\
&\leq \ \frac{2K}{\varepsilon}\E\int_0^t \Big(\big| X_s^\alpha - X_s^{\alpha'}\big|^2+\Wc_2\big(\P_{X_s^\alpha},\P_{X_s^{\alpha'}}\big)^2\Big)ds + \frac{K}{\varepsilon}\E\int_0^t \big|\Phi(\alpha)_s - \Phi(\alpha')_s\big|^2 ds \\
&+ \frac{4K^2(\beta^\varepsilon)^2}{\varepsilon^2}\E\int_0^t \Big(\big| X_s^\alpha - X_s^{\alpha'}\big|^2+\Wc_2\big(\P_{X_s^\alpha},\P_{X_s^{\alpha'}}\big)^2 + \big|\Phi(\alpha)_s - \Phi(\alpha')_s\big|^2 + \Wc_2\big(\P_{\Phi(\alpha)_s},\P_{\Phi(\alpha')_s}\big)^2\Big)ds, 
\end{align*}
where the last inequality follows from $xy\leq x^2/2+y^2/2$, for every $x,y\in\R$. Hence
\begin{align*}
\E\big|\Phi(\alpha)_t - \Phi(\alpha')_t\big|^2 \ &\leq \ \bigg(\frac{4K}{\eps}+\frac{8K^2(\beta^\varepsilon)^2}{\varepsilon^2}\bigg)\int_0^t \E\big| X_s^\alpha - X_s^{\alpha'}\big|^2 ds \\
&\quad \ + \bigg(\frac{K}{\eps} + \frac{8K^2(\beta^\varepsilon)^2}{\varepsilon^2}\bigg)\int_0^t \E\big|\Phi(\alpha)_s - \Phi(\alpha')_s\big|^2 ds \\
&\leq \ C_{\eps,K}\int_0^t \E\big| X_s^\alpha - X_s^{\alpha'}\big|^2 ds + C_{\eps,K}\int_0^t \E\big|\Phi(\alpha)_s - \Phi(\alpha')_s\big|^2 ds,
\end{align*}
with $C_{\eps,K}:=4K/\eps+8K^2(\beta^\eps)^2/\eps^2$. By Gronwall's inequality, we find
\[
\E\big|\Phi(\alpha)_t - \Phi(\alpha')_t\big|^2 \ \leq \ C_{\eps,K}\text{e}^{C_{\eps,K}T}\int_0^t \E\big| X_s^\alpha - X_s^{\alpha'}\big|^2 ds \ \leq \ C_{\eps,K}T\text{e}^{C_{\eps,K}T}\big\|X^\alpha-X^{\alpha'}\big\|_{\S^2}^2.
\]
Integrating with respect to $t$ between $t=0$ and $t=T$, we obtain
\[
\big\|\Phi(\alpha)-\Phi(\alpha')\big\|_{\H^2}^2 \ \leq \ C_{\eps,K}T^2\text{e}^{C_{\eps,K}T}\big\|X^\alpha-X^{\alpha'}\big\|_{\S^2}^2.
\]
By estimate \eqref{EstimateX-X'2} with $\bar t=T$, we have
\[
\big\|\Phi(\alpha) - \Phi(\alpha')\big\|_{\H^2}^2 \ \leq \ 8K^2C_{\eps,K}T^2 \textup{e}^{C_K(1+T^2)+C_{\eps,K}T}\|\alpha - \alpha'\|_{\H^2}^2.
\]
This shows that $\Phi$ is a contraction if $T$ is small enough. For a generic $T$, we proceed iteratively, proving existence and uniqueness on suitable subintervals of $[0,T]$.

\vspace{2mm}

\noindent\emph{Estimates \eqref{EstimateA} and \eqref{EstimateX}.} We begin noting that we have the following standard estimate for $X^\eps$ (in the sequel, $C>0$ denotes a constant only depending on $\lambda,K,T$, which may change from line to line)
\begin{equation*}
\E\Big[\sup_{0\leq s\leq t}|X_s^\eps|^2\Big] \ \leq \ C\,\E\bigg[|\xi|^2 + \int_0^t \big|b(s,0,\delta_0,A_s^\eps,\P_{A_s^\eps})\big|^2 ds + \int_0^t \big|\sigma(s,0,\delta_0,A_s^\eps,\P_{A_s^\eps})\big|^2 ds\bigg].
\end{equation*}
Therefore, in particular it holds that
\begin{equation}\label{EstimateProof_X^eps2}
\E\Big[\sup_{0\leq s\leq t}|X_s^\eps|^2\Big] \ \leq \ C\,\bigg(1+\E|\xi|^2 + \int_0^t \E|A_s^\eps|^2 ds\bigg).
\end{equation}
We begin applying It\^o's formula to $s\mapsto\text{e}^{\lambda s/(2\eps)}|A_t^\eps|^2$ between $s=0$ and $s=t\in[0,T]$, then we take the expectation and we obtain (here we use that the stochastic integral $t\mapsto\int_0^t\text{e}^{2\lambda s/(3\eps)}\langle A_s^\eps,\Sigma(s,X_s^\eps,\P_{X_s^\eps},A_s^\eps,\P_{A_s^\eps})dW_s\rangle$ is a true martingale; such a martingale property is a consequence of the fact that $X^\eps\in\S_d^2$, $A^\eps\in\H_k^2$, and $\Sigma$ satisfies a linear growth condition)
\begin{align*}
\E\big[|A_t^\eps|^2\big] \ &= \ \text{e}^{-\frac{\lambda}{2\eps}t}\E\big[|\eta|^2\big] + \frac{\lambda}{2\eps}\E\int_0^t \text{e}^{\frac{\lambda}{2\eps}(s-t)}|A_s^\eps|^2 ds + \frac{1}{\eps}\E\int_0^t \text{e}^{\frac{\lambda}{2\eps}(s-t)}\langle B(s,X_s^\eps,\P_{X_s^\eps},A_s^\eps,\P_{A_s^\eps}),A_s^\eps\rangle ds \\
&\quad \ + \frac{(\beta^\eps)^2}{\eps^2}\E\int_0^t \text{e}^{\frac{\lambda}{2\eps}(s-t)}\text{tr}\Big(\Sigma(s,X_s^\eps,\P_{X_s^\eps},A_s^\eps,\P_{A_s^\eps})\Sigma(s,X_s^\eps,\P_{X_s^\eps},A_s^\eps,\P_{A_s^\eps})\trans\Big) ds.
\end{align*}
By the monotonicity condition \eqref{Monotonicity}, the Lipschitz continuity of $B$ with respect to $(x,\mu)$, and the linear growth of $\Sigma$ in $(x,\mu,a,\nu)$, we get (using also \eqref{EstimateW2} and the elementary inequalities $\E|X_s^\eps|\leq\sqrt{\E|X_s^\eps|^2}$, $\E|A_s^\eps|\leq\sqrt{\E|A_s^\eps|^2}$)
\begin{align}\label{EstimateProofA^eps}
\E\big[|A_t^\eps|^2\big] \ &\leq \ \text{e}^{-\frac{\lambda}{2\eps}t}\E\big[|\eta|^2\big] + \frac{\lambda}{2\eps}\int_0^t \text{e}^{\frac{\lambda}{2\eps}(s-t)}\E|A_s^\eps|^2 ds - \frac{\lambda}{\eps}\int_0^t \text{e}^{\frac{\lambda}{2\eps}(s-t)}\E|A_s^\eps|^2 ds \notag \\
&\quad \ + \frac{C}{\eps}\int_0^t \text{e}^{\frac{\lambda}{2\eps}(s-t)} \sqrt{\E|X_s^\eps|^2} \sqrt{\E|A_s^\eps|^2} ds + \frac{1}{\eps}\E\int_0^t \text{e}^{\frac{\lambda}{2\eps}(s-t)} \langle B(s,0,\delta_0,0,\delta_0),A_s^\eps\rangle ds \notag \\
&\quad \ + C\frac{(\beta^\eps)^2}{\eps^2}\int_0^t \text{e}^{\frac{\lambda}{2\eps}(s-t)}\Big(1 + \E|X_s^\eps|^2 + \E|A_s^\eps|^2\Big) ds.
\end{align}
Now, by \eqref{EstimateProof_X^eps2}, the elementary inequality $\sqrt{a+b+c}\leq\sqrt{a}+\sqrt{b}+\sqrt{c}$, valid for every $a,b,c\geq0$, and also recalling from Assumption \ref{ass:two_scale} that $|B(s,0,\delta_0,0,\delta_0)|\leq K$, we find
\begin{align}\label{EstimateProofA^eps2}
&\frac{C}{\eps}\int_0^t \text{e}^{\frac{\lambda}{2\eps}(s-t)} \sqrt{\E|X_s^\eps|^2} \sqrt{\E|A_s^\eps|^2} ds + \frac{1}{\eps}\E\int_0^t \text{e}^{\frac{\lambda}{2\eps}(s-t)} \langle B(s,0,\delta_0,0,\delta_0),A_s^\eps\rangle ds \ \leq \notag \\
&\leq \ \frac{C}{\eps}\int_0^t \text{e}^{\frac{\lambda}{2\eps}(s-t)} \bigg(1+\sqrt{\E|\xi|^2}+\sqrt{\int_0^s \E|A_r^\eps|^2 dr}\bigg) \sqrt{\E|A_s^\eps|^2} ds \notag \\
&\leq \ \frac{C}{\eps}\sqrt{\int_0^t \text{e}^{\frac{\lambda}{2\eps}(s-t)} \bigg(1+\E|\xi|^2+\int_0^s \E|A_r^\eps|^2 dr\bigg) ds}\sqrt{\int_0^t \text{e}^{\frac{\lambda}{2\eps}(s-t)}\E|A_s^\eps|^2 ds} \notag \\
&\leq \ \frac{C^2}{2\lambda\eps}\int_0^t \text{e}^{\frac{\lambda}{2\eps}(s-t)} \bigg(1+\E|\xi|^2+\int_0^s \E|A_r^\eps|^2 dr\bigg) ds + \frac{\lambda}{2\eps}\int_0^t \text{e}^{\frac{\lambda}{2\eps}(s-t)}\E|A_s^\eps|^2 ds \notag \\
&= \ \frac{C^2}{\lambda^2}\big(1+\E|\xi|^2\big)\big(1 - \text{e}^{-\frac{\lambda}{2\eps}t}\big) + \frac{C^2}{2\lambda\eps}\int_0^t \E|A_r^\varepsilon|^2\bigg(\int_r^t \text{e}^{\frac{\lambda}{2\eps}(s-t)} ds\bigg) dr + \frac{\lambda}{2\eps}\int_0^t \text{e}^{\frac{\lambda}{2\eps}(s-t)}\E|A_s^\eps|^2 ds \notag \\
&= \ \frac{C^2}{\lambda^2}\big(1+\E|\xi|^2\big)\big(1 - \text{e}^{-\frac{\lambda}{2\eps}t}\big) + \frac{C^2}{\lambda^2}\int_0^t \E|A_r^\varepsilon|^2\big(1 - \text{e}^{\frac{\lambda}{2\eps}(r-t)}\big) dr + \frac{\lambda}{2\eps}\int_0^t \text{e}^{\frac{\lambda}{2\eps}(s-t)}\E|A_s^\eps|^2 ds \notag \\
&\leq \ \frac{C^2}{\lambda^2}\bigg(1+\E|\xi|^2+\int_0^t \E|A_s^\varepsilon|^2ds\bigg) + \frac{\lambda}{2\eps}\int_0^t \text{e}^{\frac{c\lambda}{\eps}(s-t)}\E|A_s^\eps|^2 ds,
\end{align}
where in the third inequality above we used that $ab\leq C a^2/(2\lambda)+\lambda b^2/(2C)$, for every $a,b\in\R$. Moreover, using again \eqref{EstimateProof_X^eps2}, we obtain
\begin{align}\label{EstimateProofA^eps3}
\frac{(\beta^\eps)^2}{\eps^2}\int_0^t \text{e}^{\frac{\lambda}{2\eps}(s-t)}\Big(1 + \E|X_s^\eps|^2 + \E|A_s^\eps|^2\Big) ds \ &\leq \ \frac{(\beta^\eps)^2}{\eps^2}\int_0^t \Big(1 + \E|X_s^\eps|^2 + \E|A_s^\eps|^2\Big) ds \notag \\
&\leq \ C\frac{(\beta^\eps)^2}{\eps^2}\bigg(1 + \E|\xi|^2 + \int_0^t \E|A_s^\eps|^2 ds\bigg).
\end{align}
Plugging estimates \eqref{EstimateProofA^eps2} and \eqref{EstimateProofA^eps3} into \eqref{EstimateProofA^eps} we get
\begin{align*}
\E\big[|A_t^\eps|^2\big] \ &\leq \ \text{e}^{-\frac{c\lambda}{\eps}t}\E\big[|\eta|^2\big] + C\bigg(1+\frac{(\beta^\eps)^2}{\eps^2}\bigg)\bigg(1+\E|\xi|^2+\int_0^t \E|A_s^\varepsilon|^2ds\bigg) \\
&\leq \ \E\big[|\eta|^2\big] + C\bigg(1+\frac{(\beta^\eps)^2}{\eps^2}\bigg)\bigg(1+\E|\xi|^2+\int_0^t \E|A_s^\varepsilon|^2ds\bigg).
\end{align*}
Applying Gronwall's inequality to $v(t)=\E|A_t^\varepsilon|^2$, $t\in[0,T]$, we find
\[
\E\big[|A_t^\eps|^2\big] \ \leq \ \bigg(\E\big[|\eta|^2\big] + C\bigg(1+\frac{(\beta^\eps)^2}{\eps^2}\bigg)\big(1+\E|\xi|^2\big)\bigg)\text{e}^{C\left(1+\frac{(\beta^\eps)^2}{\eps^2}\right)T}, \qquad \forall\,t\in[0,T].
\]
which yields \eqref{EstimateA}. Finally, plugging \eqref{EstimateA} into \eqref{EstimateProof_X^eps2} we get \eqref{EstimateX}.
\end{proof}

\begin{Assumption}{\bf(B)}\label{ass:optimal_control}
For every $\xi\in L^2(\Omega,\Gc,\P;\R^d)$ there exists $\widehat\alpha\in\H_k^2$ such that
\begin{equation}\label{widehat_alpha}
\big\|B\big(\cdot,\widehat X_\cdot,\P_{\widehat X_\cdot},\widehat\alpha_\cdot,\P_{\widehat\alpha_\cdot}\big)\big\|_{\H^2}^2 \ = \ \E\bigg[\int_0^T \big|B\big(t,\widehat X_t,\P_{\widehat X_t},\widehat\alpha_t,\P_{\widehat\alpha_t}\big)\big|^2\,dt\bigg] \ = \ 0,
\end{equation}
where $\widehat X$ is the unique solution in $\S_d^2$ to the following controlled stochastic differential equation on $[0,T]$:
\begin{equation}\label{widehat_X}
d\widehat X_t \ = \ b\big(t,\widehat X_t,\P_{\widehat X_t},\widehat\alpha_t,\P_{\widehat\alpha_t}\big)dt+\sigma\big(t,\widehat X_t,\P_{\widehat X_t},\widehat\alpha_t,\P_{\widehat\alpha_t}\big)dW_t, \quad t\in[0,T], \qquad\quad \widehat X_0 \ = \ \xi.
\end{equation}
\end{Assumption}

\begin{proposition}\label{Prop:Uniqueness}
Let Assumptions \ref{ass:two_scale} and \ref{ass:optimal_control} hold. Then, there exists a unique pair of processes $(\widehat X,\widehat\alpha)\in\S_d^2\times\H_k^2$ satisfying \eqref{widehat_alpha}-\eqref{widehat_X}.
\end{proposition}
\begin{proof}
Let $(X',\alpha')\in\S_d^2\times\H_k^2$ be another pair satisfying \eqref{widehat_alpha}-\eqref{widehat_X}. First observe that if $\|\widehat\alpha-\alpha'\|_{\H^2}=0$, then by estimate \eqref{EstimateX-X'2} we have $\|\widehat X-X'\|_{\S^2}=0$. Therefore, it is enough to prove $\|\widehat\alpha-\alpha'\|_{\H^2}=0$. To this end, notice that, for any $\bar t\in[0,T]$, it holds that
\begin{align*}
&\bigg|\E\bigg[\int_0^{\bar t} \big\langle B\big(t,\widehat X_t,\P_{\widehat X_t},\widehat\alpha_t,\P_{\widehat\alpha_t}\big) - B\big(t,X_t',\P_{X_t'},\alpha_t',\P_{\alpha_t'}\big),\widehat\alpha_t - \alpha_t'\big\rangle dt\bigg]\bigg| \\
&\leq \ \E\bigg[\int_0^T \Big(\big|B\big(t,\widehat X_t,\P_{\widehat X_t},\widehat\alpha_t,\P_{\widehat\alpha_t}\big)\big| + \big|B\big(t,X_t',\P_{X_t'},\alpha_t',\P_{\alpha_t'}\big)\big|\widehat\alpha_t - \alpha_t'\big|dt\bigg] \\
&\leq \ \Big(\big\|B\big(\cdot,\widehat X_\cdot,\P_{\widehat X_\cdot},\widehat\alpha_\cdot,\P_{\widehat\alpha_\cdot}\big)\big\|_{\H^2} + \big\|B\big(\cdot,X_\cdot',\P_{X_\cdot'},\alpha_\cdot',\P_{\alpha_\cdot'}\big)\big\|_{\H^2}\Big)\big\|\widehat\alpha - \alpha'\big\|_{\H^2} \ = \ 0.
\end{align*}
Hence
\begin{align*}
0 \ &= \ \E\bigg[\int_0^{\bar t} \big\langle B\big(t,\widehat X_t,\P_{\widehat X_t},\widehat\alpha_t,\P_{\widehat\alpha_t}\big) - B\big(t,X_t',\P_{X_t'},\alpha_t',\P_{\alpha_t'}\big),\widehat\alpha_t - \alpha_t'\big\rangle dt\bigg] \notag \\
&= \ \E\bigg[\int_0^{\bar t} \big\langle B\big(t,\widehat X_t,\P_{\widehat X_t},\widehat\alpha_t,\P_{\widehat\alpha_t}\big) - B\big(t,\widehat X_t,\P_{\widehat X_t},\alpha_t',\P_{\alpha_t'}\big),\widehat\alpha_t - \alpha_t'\big\rangle dt\bigg] \notag \\
&\quad \ + \E\bigg[\int_0^{\bar t} \big\langle B\big(t,\widehat X_t,\P_{\widehat X_t},\alpha_t',\P_{\alpha_t'}\big) - B\big(t,X_t',\P_{X_t'},\alpha_t',\P_{\alpha_t'}\big),\widehat\alpha_t - \alpha_t'\big\rangle dt\bigg] \notag \\
&\leq \ -\lambda\E\bigg[\int_0^{\bar t} |\widehat\alpha_t - \alpha_t'|^2 dt\bigg] + K\E\bigg[\int_0^{\bar t}\Big(|\widehat X_t - X_t'| + \Wc_2\big(\P_{\widehat X_t},\P_{X_t'}\big)\Big)|\widehat\alpha_t-\alpha_t'|dt\bigg],
\end{align*}
where the last inequality follows from the monotonicity and Lipschitz properties of $B$. Now, from the Cauchy-Schwarz inequality and estimate \eqref{EstimateW2}, we find
\begin{align*}
\lambda\E\bigg[\int_0^{\bar t} |\widehat\alpha_t - \alpha_t'|^2 dt\bigg] \ &\leq \ K\sqrt{\E\bigg[\int_0^{\bar t}\Big(|\widehat X_t - X_t'| + \Wc_2\big(\P_{\widehat X_t},\P_{X_t'}\big)\Big)^2dt\bigg]}\sqrt{\E\bigg[\int_0^{\bar t}|\widehat\alpha_t-\alpha_t'|^2dt\bigg]} \\
&\leq \ \sqrt{2}K\sqrt{\E\bigg[\int_0^{\bar t}|\widehat X_t - X_t'|^2 dt\bigg]}\sqrt{\E\bigg[\int_0^{\bar t}|\widehat\alpha_t-\alpha_t'|^2dt\bigg]} \\
&\leq \ \sqrt{2\bar t}K\sqrt{\E\Big[\sup_{0\leq t\leq\bar t}|\widehat X_t - X_t'|^2\Big]}\sqrt{\E\bigg[\int_0^{\bar t}|\widehat\alpha_t-\alpha_t'|^2dt\bigg]}.
\end{align*}
Now, by estimate \eqref{EstimateX-X'2} we obtain
\begin{equation}\label{alpha-alpha'}
\lambda\E\bigg[\int_0^{\bar t} |\widehat\alpha_t - \alpha_t'|^2 dt\bigg] \ \leq \ 4\sqrt{\bar t}\,K^2\textup{e}^{\frac{C_K(1+T^2)}{2}}\E\bigg[\int_0^{\bar t} \big|\widehat \alpha_t - \alpha_t'\big|^2 dt\bigg].
\end{equation}
Let $F\colon[0,T]\rightarrow[0,\infty)$ be given by
\[
F(t) \ = \ \E\bigg[\int_0^t \big|\widehat \alpha_s - \alpha_s'\big|^2 dt\bigg], \qquad \forall\,t\in[0,T].
\]
Notice that $F(0)=0$ and $F$ is a monotone non-decreasing function. Since $F(T)=\|\widehat\alpha-\alpha'\|_{\H^2}$, we see that the claim follows if we prove that $F$ is constant. This is indeed a direct consequence of \eqref{alpha-alpha'}, which written in terms of $F$ becomes
\begin{equation}\label{alpha-alpha'_F}
F(t) \ \leq \ \hat C\sqrt{t}\,F(t), \qquad \text{for all }t\in[0,T],
\end{equation}
with $\hat C:=(4K^2/\lambda)\exp(C_K(1+T^2)/2)>0$. Since $F\geq0$, from \eqref{alpha-alpha'_F} we deduce that
\[
\hat C\sqrt{t}\,<\,1 \ \ \ \Longrightarrow \ \ \ F(t)\,=\,0.
\]
This means that $F$ is constant on the interval $[0,\ell)$, where $\ell:=\hat C^{-2}$. Therefore, starting for instance from the interval $[0,\ell/2]$, then considering the interval $[\ell/2,\ell]$, afterwards the interval $[\ell,3\ell/2]$, and so on, we see that after a finite number of steps we conclude that $F$ is constant on the entire interval $[0,T]$.
\end{proof}

\begin{lemma}\label{L:Density}
The family
\[
\cH \ = \ \bigg\{\alpha\in\H_k^2\colon\alpha_t=\int_0^tF_s\,ds,\;\forall\,t\in[0,T],\text{ for some }F\in\H_k^2\bigg\}
\]
is dense in $\H_k^2$.

\end{lemma}
\begin{proof}
Consider the family
\[
\cS \ = \ \bigg\{\alpha\in\H_k^2\colon\alpha_t=\xi1_{(t_0,T]}(t),\;\forall\,t\in[0,T],\text{ for some }t_0\in[0,T),\,\xi\colon\Omega\rightarrow\R^k,\;\cF_{t_0}\text{-meas. and bounded}\bigg\}.
\]
It is well-known that linear space generated by $\cS$ is dense in $\H_k^2$, see for instance \cite[Lemma 3.2.4]{KaratzasShreve} (notice that it coincides with the linear space generated by the processes of the form $\xi1_{(t_0,t_1]}$, with $t_0,t_1\in[0,T]$, $t_0<t_1$, and $\xi$ being $\cF_{t_0}$-measurable and bounded). Then, it is enough to prove that, for any $t_0\in[0,T)$ and $\xi\colon\Omega\rightarrow\R^k$, with $\xi$ being $\cF_{t_0}$-measurable and bounded, there exists a sequence $(F^n)_n\subset\H_k^2$ such that
\[
\E\bigg[\int_0^T \bigg|\int_0^t F_s^n\,ds - \xi1_{(t_0,T]}(t)\bigg|^2 dt\bigg] \ \underset{n\rightarrow\infty}{\longrightarrow} \ 0.
\]
Let $F_s^n=n\xi1_{[t_0,t_0+1/n]}(s)$, $s\in[0,T]$, $n\in\N$. Then
\[
\E\bigg[\int_0^T \bigg|\int_0^t F_s^n\,ds - \xi1_{(t_0,T]}(t)\bigg|^2 dt\bigg] \ = \ \E\bigg[\int_{t_0}^{t_0+\frac{1}{n}} \big|n\xi(t-t_0) - \xi\big|^2 dt\bigg] \ = \ \frac{1}{3n} \E\big[|\xi|^2\big] \ \underset{n\rightarrow\infty}{\longrightarrow} \ 0.
\]
\end{proof}

\noindent We can now state our first formulation of the stochastic Tikhonov theorem, which applies to the two-scale system \eqref{TwoScale}.

\begin{theorem}[Stochastic Tikhonov Theorem I]\label{T:Convergence}
Let Assumptions \ref{ass:two_scale} and \ref{ass:optimal_control} hold. Moreover, let $(\widehat X,\widehat\alpha)\in\S_d^2\times\H_k^2$ be as in Assumption \ref{ass:optimal_control} and, for every $\varepsilon>0$, let $(X^\varepsilon,A^\varepsilon)\in\S_d^2\times\H_k^2$ be as in Proposition \ref{P:System_eps}. Then, it holds that
\begin{equation}\label{Convergence_A}
\big\|A^\varepsilon - \widehat\alpha\big\|_{\H^2} \ \underset{\varepsilon\rightarrow0}{\longrightarrow} \ 0
\end{equation}
and
\begin{equation}\label{Convergence_X}
\big\|X^\varepsilon - \widehat X\big\|_{\S^2} \ \underset{\varepsilon\rightarrow0}{\longrightarrow} \ 0.
\end{equation}
\end{theorem}
\begin{proof}
We begin by noticing that \eqref{Convergence_X} follows directly from \eqref{Convergence_A} and estimate \eqref{EstimateX-X'2} (with $\bar t=T$). Thus, we only need to prove \eqref{Convergence_A}. Let $\widehat\alpha\in\H_k^2$ be the process appearing in Assumption \ref{ass:optimal_control}. By Lemma \ref{L:Density} there exists $(F^n)_n\in\H_k^2$ such that $\|\alpha^n-\widehat\alpha\|_{\H^2}\rightarrow0$ as $n\rightarrow\infty$, with $\alpha_t^n=\int_0^t F_s^n ds$, $t\in[0,T]$ (as in \eqref{tilde_alpha}, we redefine $\alpha^n$ at $t=0$ as: $\alpha_0^n=\eta$). Let also $X^n\in\S_d^2$ denote the solution to the following equation:
\[
dX_t^n \ = \ b\big(t,X_t^n,\P_{X_t^n},\alpha_t^n,\P_{\alpha_t^n}\big)dt+\sigma\big(t,X_t^n,\P_{X_t^n},\alpha_t^n,\P_{\alpha_t^n}\big)dW_t, \qquad t\in[0,T],
\]
with $X_0^n=\xi$. Applying It\^o's formula to $\text{e}^{\frac{2\lambda}{3\varepsilon} t}|A_t^\varepsilon - \alpha_t^n|^2$, and taking the expectation (here we use that the stochastic integral $t\mapsto\int_0^t\text{e}^{2\lambda s/(3\eps)}\langle A_s^\eps-\alpha_s^n,\Sigma(s,X_s^\eps,\P_{X_s^\eps},A_s^\eps,\P_{A_s^\eps})dW_s\rangle$ is a true martingale; such a martingale property is a consequence of the fact that $X^\eps\in\S_d^2$, $A^\eps,\alpha^n\in\H_k^2$, and $\Sigma$ satisfies a linear growth condition), we find
\begin{align*}
\E|A_t^\varepsilon - \alpha_t^n|^2 &= \frac{2\lambda}{3\varepsilon}\int_0^t \text{e}^{\frac{2\lambda}{3\varepsilon}(s-t)} \E|A_s^\varepsilon - \alpha_s^n|^2 ds + \frac{2}{\varepsilon}\int_0^t \text{e}^{\frac{2\lambda}{3\varepsilon}(s-t)} \E\big\langle B\big(s,X_s^\varepsilon,\P_{X_s^\eps},A_s^\varepsilon,\P_{A_s^\eps}\big),A_s^\varepsilon - \alpha_s^n\big\rangle ds \\
&\quad \ - 2\int_0^t \text{e}^{\frac{2\lambda}{3\varepsilon}(s-t)} \E\big\langle F_s^n,A_s^\varepsilon - \alpha_s^n\big\rangle ds + \frac{(\beta^\varepsilon)^2}{\varepsilon^2} \int_0^t \text{e}^{\frac{2\lambda}{3\varepsilon}(s-t)} \E\big|\Sigma\big(s,X_s^\varepsilon,\P_{X_s^\eps},A_s^\varepsilon,\P_{A_s^\eps}\big)\big|^2 ds,
\end{align*}
where we can interchange the expectation and the integral with respect to time thanks to Fubini's theorem, as a matter of fact the integrands belong to $L^1([0,t]\times\Omega,\Bc([0,t])\otimes\Fc,ds\otimes d\P)$, for every $t\in[0,T]$. Proceeding along the same lines as for \eqref{EstimateProofA^eps}, we find (in the sequel, $C>0$ denotes a constant only depending on $\lambda,K,T$, which may change from line to line)
\begin{equation*}
v(t) \ \leq \ C\int_0^t v(s)\,ds + u(t),
\end{equation*}
where $v(t)=\E|A_t^\varepsilon-\alpha_t^n|^2$, $t\in[0,T]$, and
\begin{align*}
u(t) \ &= \ \frac{C}{\varepsilon} \int_0^t \text{e}^{\frac{2\lambda}{3\varepsilon}(s-t)} \E\big|B(s,X_s^n,\P_{X_s^n},\alpha_s^n,\P_{\alpha_s^n})\big|^2 ds + 4\int_0^t\text{e}^{\frac{2\lambda}{3\varepsilon}(s-t)}\E\big|F_s^n\big|^2ds \\
&\quad \ + \frac{(\beta^\varepsilon)^2}{\varepsilon^2} \int_0^t \text{e}^{\frac{2\lambda}{3\varepsilon}(s-t)} \E\big|\Sigma(s,X_s^\varepsilon,\P_{X_s^\eps},A_s^\varepsilon,\P_{A_s^\eps})\big|^2 ds.
\end{align*}
Applying Gronwall's inequality, we obtain $v(t)\leq u(t)+\int_0^t u(r)C\text{e}^{C(t-r)}dr$. Possibly enlarging $C$, from the latter inequality we find $v(t)\leq u(t)+C\int_0^t u(r)dr$, that is
\begin{align}\label{ProofLimit4}
&\E|A_t^\varepsilon - \alpha_t^n|^2 \ \leq \ \frac{C}{\varepsilon} \int_0^t \text{e}^{\frac{2\lambda}{3\varepsilon}(s-t)} \E\big|B(s,X_s^n,\P_{X_s^n},\alpha_s^n,\P_{\alpha_s^n})\big|^2 ds + 4\int_0^t\text{e}^{\frac{2\lambda}{3\varepsilon}(s-t)}\E\big|F_s^n\big|^2ds \notag \\
&+ \frac{(\beta^\varepsilon)^2}{\varepsilon^2} \int_0^t \text{e}^{\frac{2\lambda}{3\varepsilon}(s-t)} \E\big|\Sigma(s,X_s^\varepsilon,\P_{X_s^\eps},A_s^\varepsilon,\P_{A_s^\eps})\big|^2 ds \notag \\
&+ C\bigg\{\frac{C}{\varepsilon}\int_0^t \bigg(\int_0^r \text{e}^{\frac{2\lambda}{3\varepsilon}(s-r)} \E\big|B(s,X_s^n,\P_{X_s^n},\alpha_s^n,\P_{\alpha_s^n})\big|^2 ds\bigg)dr + 4\int_0^t\bigg(\int_0^r\text{e}^{\frac{2\lambda}{3\varepsilon}(s-r)}\E\big|F_s^n\big|^2ds\bigg)dr \notag \\
&+ \frac{(\beta^\varepsilon)^2}{\varepsilon^2} \int_0^t\bigg(\int_0^r \text{e}^{\frac{2\lambda}{3\varepsilon}(s-r)} \E\big|\Sigma(s,X_s^\varepsilon,\P_{X_s^\eps},A_s^\varepsilon,\P_{A_s^\eps})\big|^2 ds\bigg)dr\bigg\}.
\end{align}
Notice that
\begin{align*}
\int_0^t\bigg(\int_0^r\text{e}^{\frac{2\lambda}{3\varepsilon}(s-r)}\E\big|F_s^n\big|^2ds\bigg)dr \ &= \ \int_0^t\E\big|F_s^n\big|^2\bigg(\int_s^t\text{e}^{\frac{2\lambda}{3\varepsilon}(s-r)}dr\bigg)ds \\
&= \ \frac{3\eps}{2\lambda}\int_0^t\E\big|F_s^n\big|^2\Big(1-\textup{e}^{\frac{2\lambda}{3\varepsilon}(s-t)}\Big)ds \ \leq \ C\eps\int_0^t\E\big|F_s^n\big|^2ds.
\end{align*}
An analogous estimate holds for the terms concerning $B$ and $\Sigma$. Therefore, by \eqref{ProofLimit4} we obtain
\begin{align*}
&\E|A_t^\varepsilon - \alpha_t^n|^2 \ \leq \ \frac{C}{\eps}\int_0^t \Big(\text{e}^{\frac{2\lambda}{3\varepsilon}(s-t)}+\eps\Big) \E\big|B(s,X_s^n,\P_{X_s^n},\alpha_s^n,\P_{\alpha_s^n})\big|^2 ds \\
&+ C\int_0^t\Big(\text{e}^{\frac{2\lambda}{3\varepsilon}(s-t)}+\eps\Big)\E\big|F_s^n\big|^2ds + C\frac{(\beta^\varepsilon)^2}{\varepsilon^2} \int_0^t \Big(\text{e}^{\frac{2\lambda}{3\varepsilon}(s-t)}+\eps\Big) \E\big|\Sigma(s,X_s^\varepsilon,\P_{X_s^\eps},A_s^\varepsilon,\P_{A_s^\eps})\big|^2 ds.
\end{align*}
Integrating with respect to $t$ between $t=0$ and $t=T$, we find
\begin{align}\label{ProofLimit2}
&\big\|A^\varepsilon - \alpha^n\big\|_{\H^2}^2 \ \leq \ \frac{C}{\varepsilon}\!\!\int_0^T\!\!\!\!\int_0^t\!\! \Big(\text{e}^{\frac{2\lambda}{3\varepsilon}(s-t)}+\eps\Big) \E\big|B(s,X_s^n,\P_{X_s^n},\alpha_s^n,\P_{\alpha_s^n})\big|^2 dsdt \\
&+ C\!\!\int_0^T\!\!\!\!\int_0^t\!\!\Big(\text{e}^{\frac{2\lambda}{3\varepsilon}(s-t)}+\eps\Big)\E\big|F_s^n\big|^2dsdt + C\frac{(\beta^\varepsilon)^2}{\varepsilon^2}\!\! \int_0^T\!\!\!\!\int_0^t\!\! \Big(\text{e}^{\frac{2\lambda}{3\varepsilon}(s-t)}+\eps\Big) \E\big|\Sigma(s,X_s^\varepsilon,\P_{X_s^\eps},A_s^\varepsilon,\P_{A_s^\eps})\big|^2 dsdt. \notag
\end{align}
Letting $\varepsilon\rightarrow0$, recalling estimate \eqref{EstimateX}, the growth condition of $\Sigma$, equality \eqref{NormWass}, and the fact that $\beta^\varepsilon=o(\varepsilon)$ as $\eps\rightarrow0$, we see that
\begin{equation}\label{Convergence_Sigma}
\frac{(\beta^\varepsilon)^2}{\varepsilon^2} \int_0^T\int_0^t \Big(\text{e}^{\frac{2\lambda}{3\varepsilon}(s-t)}+\eps\Big) \E\big|\Sigma(s,X_s^\varepsilon,\P_{X_s^\varepsilon},A_s^\varepsilon,\P_{A_s^\eps})\big|^2 dsdt \ \underset{\varepsilon\rightarrow0}{\longrightarrow} \ 0.
\end{equation}
Similarly, since $F^n\in\H_k^2$ and $F^n$ is independent of $\eps$, by Lebesgue's dominated convergence theorem we deduce that
\begin{equation}\label{Convergence_Fn}
\int_0^T\int_0^t\Big(\text{e}^{\frac{2\lambda}{3\varepsilon}(s-t)}+\eps\Big)\E\big|F_s^n\big|^2dsdt \ \underset{\varepsilon\rightarrow0}{\longrightarrow} \ 0.
\end{equation}
Finally, we have
\begin{align*}
&\frac{1}{\varepsilon} \int_0^T\int_0^t \Big(\text{e}^{\frac{2\lambda}{3\varepsilon}(s-t)}+\eps\Big) \E\big|B(s,X_s^n,\P_{X_s^n},\alpha_s^n,\P_{\alpha_s^n})\big|^2 dsdt \\
&= \ \frac{1}{\varepsilon} \int_0^T\E\big|B(s,X_s^n,\P_{X_s^n},\alpha_s^n,\P_{\alpha_s^n})\big|^2\bigg(\int_s^T \Big(\text{e}^{\frac{2\lambda}{3\varepsilon}(s-t)}+\eps\Big)dt\bigg)ds \\
&= \ \frac{1}{\varepsilon} \int_0^T\E\big|B(s,X_s^n,\P_{X_s^n},\alpha_s^n,\P_{\alpha_s^n})\big|^2\bigg(-\frac{3\eps}{2\lambda}\text{e}^{\frac{2\lambda}{3\varepsilon}(s-t)}\Big|_{t=s}^{t=T}+\eps(T-s)\bigg)ds \\
&= \ \frac{1}{\varepsilon} \int_0^T\E\big|B(s,X_s^n,\P_{X_s^n},\alpha_s^n,\P_{\alpha_s^n})\big|^2\bigg(\frac{3\eps}{2\lambda} - \frac{3\eps}{2\lambda}\text{e}^{\frac{2\lambda}{3\varepsilon}(s-T)}+\eps(T-s)\bigg)ds \\
&\leq \ \frac{1}{\varepsilon} \int_0^T\E\big|B(s,X_s^n,\P_{X_s^n},\alpha_s^n,\P_{\alpha_s^n})\big|^2\bigg(\frac{3\eps}{2\lambda}+\eps T\bigg)ds \\
&= \ \bigg(\frac{3}{2\lambda}+T\bigg)\int_0^T\E\big|B(s,X_s^n,\P_{X_s^n},\alpha_s^n,\P_{\alpha_s^n})\big|^2ds \ = \ \bigg(\frac{3}{2\lambda}+T\bigg)\big\|B\big(\cdot,X_\cdot^n,\P_{X_\cdot^n},\alpha_\cdot^n,\P_{\alpha_\cdot^n}\big)\big\|_{\H^2}^2.
\end{align*}
Hence, by the above limit and also by \eqref{Convergence_Sigma} and \eqref{Convergence_Fn}, sending $\varepsilon\rightarrow0$ in \eqref{ProofLimit2} we obtain
\[
\limsup_{\varepsilon\rightarrow0} \big\|A^\varepsilon - \alpha^n\big\|_{\H^2}^2 \ \leq \ C\big\|B\big(\cdot,X_\cdot^n,\P_{X_\cdot^n},\alpha_\cdot^n,\P_{\alpha_\cdot^n}\big)\big\|_{\H^2}^2.
\]
Now, notice that
\[
\big\|A^\varepsilon - \widehat\alpha\big\|_{\H^2} \ \leq \ \big\|A^\varepsilon - \alpha^n\big\|_{\H^2} + \big\|\alpha^n - \widehat\alpha\big\|_{\H^2}.
\]
Hence, sending $\varepsilon\rightarrow0$ we find
\begin{equation}\label{ProofLimit3}
\limsup_{\varepsilon\rightarrow0}\big\|A^\varepsilon - \widehat\alpha\big\|_{\H^2} \ \leq \ C\big\|B\big(\cdot,X_\cdot^n,\P_{X_\cdot^n},\alpha_\cdot^n,\P_{\alpha_\cdot^n}\big)\big\|_{\H^2} + \big\|\alpha^n - \widehat\alpha\big\|_{\H^2}, \qquad \forall\,n\in\N.
\end{equation}
Recalling that $\alpha^n\rightarrow\widehat\alpha$ in $\H_k^2$ as $n\rightarrow\infty$, we deduce from estimate \eqref{EstimateX-X'2} (with $\bar t=T$) that $X^n\rightarrow\widehat X$ in $\S_d^2$. This in turn implies that there exists a subsequence $\{(X^{n_m},\alpha^{n_m})\}_m$ which converges to $(\widehat X,\widehat\alpha)$ pointwise. As a consequence, from the continuity of $B$, we deduce that $\{B(\cdot,X_\cdot^{n_m},\P_{X_\cdot^{n_m}},\alpha_\cdot^{n_m},\P_{\alpha_\cdot^{n_m}})\}_m$ also converges pointwise $\P\otimes dt$-a.e. to $B(\cdot,\widehat X_\cdot,\P_{\widehat X_\cdot},\widehat\alpha_\cdot,\P_{\widehat\alpha_\cdot})$, which is equal to zero. Then, from estimate \eqref{EstimateX}, the growth condition of $B$, and Lebesgue's dominated convergence theorem, it follows that $\|B(\cdot,X_\cdot^{n_m},\P_{X_\cdot^{n_m}},\alpha_\cdot^{n_m},\P_{\alpha_\cdot^{n_m}})\|_{\H^2}$ $\rightarrow0$. In conclusion, from \eqref{ProofLimit3}, and in particular from the arbitrariness of $n\in\N$, we obtain
\[
\limsup_{\varepsilon\rightarrow0}\big\|A^\varepsilon - \widehat\alpha\big\|_{\H^2} \ = \ 0,
\]
which proves \eqref{Convergence_A}.
\end{proof}

\subsection{Stochastic Tikhonov theorem II}
\label{subsec:TikI}

\noindent We present in this section a formulation of the stochastic Tikhonov theorem which applies to a class of coupled forward-backward systems which arise in the study of mean field control problems, as described in Section \ref{sec:control}. 
More precisely, given $\xi\in L^2(\Omega,\Gc,\P;\R^d)$, $\eta\in L^2(\Omega,\Gc,\P;\R^k)$, we consider, for every $\varepsilon>0$, the following coupled two-scale forward-backward system of stochastic differential equations on $[0,T]$:
\begin{equation}\label{TwoScaleFB}
	\left\{\begin{aligned}dX^\varepsilon_t&=b\big(t,X^\varepsilon_t,\P_{X_t^\varepsilon},A_t^\varepsilon,\P_{A_t^\varepsilon}\big)dt+\sigma\big(t,X^\varepsilon_t,\P_{X_t^\varepsilon},A_t^\varepsilon,\P_{A_t^\varepsilon}\big) dW_t,\\
	dY_t^\eps&=-\Phi(t,X_t^\eps,Y_t^\eps,Z_t^\eps,A_t^\eps,\P_{(X_t^\eps,Y_t^\eps,Z_t^\eps,A_t^\eps)})dt + Z_t^\eps dW_t, \\
	d\bar Y_t^\eps&=-\Psi(t,X_t^\eps,Y_t^\eps,Z_t^\eps,A_t^\eps,\P_{(X_t^\eps,Y_t^\eps,Z_t^\eps,A_t^\eps)})dt + \bar Z_t^\eps dW_t, \\
		\varepsilon dA^\varepsilon_t&=-\bar Y_t^\eps dt+\beta^\varepsilon \Gamma dW_t,\\
		X^\varepsilon_0&=\xi,\ Y_T^\eps=\phi(X_T^\eps,\P_{X_T^\eps}),\ \bar Y_T^\eps=0,\ A^\varepsilon_0=\eta.
	\end{aligned}\right.
\end{equation}
On the coefficients $b,\sigma$ we impose Assumption \ref{ass:two_scale}, moreover $\beta^\varepsilon=o(\varepsilon)$ as $\eps\rightarrow0$ and $\Gamma$ is a matrix of size $k\times m$. On the other hand, the coefficients
\begin{align*}
\Phi\,,\,\Psi \ \colon[0,T]\times\R^d\times\R^d\times\R^{d\times m}\times\R^k\times\Pc_2(\R^d\times\R^d\times\R^{d\times m}\times\R^k) \ \ \ &\longrightarrow \ \ \ \R^d\,,\,\R^k, \\
\phi \ \colon\R^d\times\Pc_2(\R^d) \ \ \ &\longrightarrow \ \ \ \R^d
\end{align*}
we impose the following assumptions.

\begin{Assumption}{\bf(C)}\label{ass:two_scaleFB}\quad
\begin{enumerate}[\upshape1)]
\item $\Phi$, $\Psi$, $\phi$ are Borel measurable functions.
\item $\Phi$, $\Psi$, $\phi$ satisfy linear growth conditions: there exists a constant $K>0$ such that
\[
|\Phi(t,x,y,z,a,\pi)| + |\Psi(t,x,y,z,a,\pi)| + |\phi(x,\mu)| \leq K\big(1 + |x| + \|\mu\|_2 + |y| + |z| + |a| + \|\pi\|_2\big),
\]
for all $(t,x,\mu,y,z,a,\pi)\in[0,T]\times\R^d\times\Pc_2(\R^d)\times\R^d\times\R^{d\times m}\times\R^k\times\Pc(\R^d\times\R^d\times\R^{d\times m}\times\R^k)$.
\item $\Phi$, $\Psi$, $\phi$ satisfy the Lipschitz continuity condition: there exists a constant $K>0$ such that
\begin{align*}
&|\Phi(t,x,y,z,a,\pi) - \Phi(t,x',y',z',a',\pi')| + |\Psi(t,x,y,z,a,\pi) - \Psi(t,x',y',z',a',\pi')| \\
&+ \, |\phi(x,\mu) - \phi(x',\mu')| \leq K\big(|x-x'|+\Wc_2(\mu,\mu')+|y-y'|+|z-z'|+|a-a'|+\Wc_2(\pi,\pi')\big),
\end{align*}
for all $t\in[0,T]$, $(x,\mu,y,z,a,\pi),(x',\mu',y',z',a',\pi')\in\R^d\times\Pc_2(\R^d)\times\R^d\times\R^{d\times m}\times\R^k\times\Pc_2(\R^d\times\R^d\times\R^{d\times m}\times\R^k)$.
\item The following monotonicity conditions hold: there exists a constant $\lambda>0$ such that
\begin{align*}
&\E\Big[\big\langle b\big(t,\xi,\P_\xi,\eta,\P_\eta\big) - b\big(t,\xi',\P_{\xi'},\eta',\P_{\eta'}\big),\upsilon - \upsilon'\big\rangle\Big] \notag \\
&+ \E\Big[\textup{tr}\Big(\big(\sigma\big(t,\xi,\P_\xi,\eta,\P_\eta\big) - \sigma\big(t,\xi',\P_{\xi'},\eta',\P_{\eta'}\big)\big)\big(\zeta - \zeta'\big)\trans\Big)\Big] \notag \\
&- \E\Big[\big\langle\Phi\big(t,\xi,\upsilon,\zeta,\eta,\P_{(\xi,\upsilon,\zeta,\eta)}\big) - \Phi\big(t,\xi',\upsilon',\zeta',\eta',\P_{(\xi',\upsilon',\zeta',\eta')}\big),\xi - \xi'\big\rangle\Big] \notag \\
&- \E\Big[\big\langle\Psi\big(t,\xi,\upsilon,\zeta,\eta,\P_{(\xi,\upsilon,\zeta,\eta)}\big) - \Psi\big(t,\xi',\upsilon',\zeta',\eta',\P_{(\xi',\upsilon',\zeta',\eta')}\big),\eta - \eta'\big\rangle\Big] \leq - \lambda\E\big[|\eta - \eta'|^2\big]
\end{align*}
and
\[
\E\big[\big\langle \phi\big(\xi,\P_\xi\big) - \phi\big(\xi',\P_{\xi'}\big),\xi - \xi'\big\rangle\big] \ \geq \ 0,
\]
for all $t\in[0,T]$, $\xi,\xi',\upsilon,\upsilon'\in L^2(\Omega,\Fc,\P;\R^d)$, $\eta,\eta'\in L^2(\Omega,\Fc,\P;\R^k)$, $\zeta,\zeta'\in L^2(\Omega,\Fc,\P;\R^{d\times m})$.
\end{enumerate}
\end{Assumption}

\begin{proposition}\label{P:System_eps_FB}
Let Assumption \ref{ass:two_scaleFB} hold. For every $\varepsilon>0$, there exists a unique solution $(X^\varepsilon,Y^\eps,Z^\eps,\bar Y^\eps,\bar Z^\eps,A^\varepsilon)$ in $\S_d^2\times\S_d^2\times\H_{d\times m}^2\times\S_k^2\times\H_{k\times m}^2\times\S_k^2$ to system \eqref{TwoScaleFB}. Moreover, there exists a constant $C$, independent of $\eps$, such that
\begin{align}\label{Estimate_FB}
&\|X^\eps\|_{\S^2}^2 + \|A^\eps\|_{\H^2}^2 + \|Y^\eps\|_{\S^2}^2 + \|Z^\eps\|_{\H^2}^2 + \|\bar Y^\eps\|_{\S^2}^2 + \|\bar Z^\eps\|_{\H^2}^2 \notag \\
&\leq \ C\,\bigg(\E|\xi|^2 + \E|\eta|^2 + |\phi(0,\delta_0)|^2 + \int_0^T\big|b(t,0,\delta_0,0,\delta_0)\big|^2 dt + \int_0^T\big|\sigma(t,0,\delta_0,0,\delta_0)\big|^2 dt \\
&\quad \ + \int_0^T\big|\Phi(t,0,0,0,0,\delta_0)\big|^2 dt + \int_0^T\big|\Psi(t,0,0,0,0,\delta_0)\big|^2 dt + \frac{(\beta^\eps)^2}{\eps^2}\textup{tr}\big(\Gamma\Gamma\trans\big)\bigg) \notag
\end{align}
\end{proposition}
\begin{proof}
The existence and uniqueness result follows from \cite[Corollary 2.4]{RSZ20_path_reg}. Regarding estimate \eqref{Estimate_FB}, we begin noting that, by standard estimates for backward stochastic differential 
equations, we have, for some constant $C$, independent of $\eps$,
\begin{equation}\label{EstimateProof1}
\|Y^\eps\|_{\S^2}^2 + \|Z^\eps\|_{\H^2} \ \leq \ C\,\E\bigg[\big|\phi(X_T^\eps,\P_{X_T^\eps})\big|^2 + \int_0^T \big|\Phi(t,X_t^\eps,0,0,A_t^\eps,\P_{(X_t^\eps,0,0,A_t^\eps)})\big|^2 dt \bigg]
\end{equation}
and
\begin{equation}\label{EstimateProof2}
\|\bar Y^\eps\|_{\S^2}^2 + \|\bar Z^\eps\|_{\H^2} \ \leq \ C\,\E\bigg[\int_0^T \big|\Psi(t,X_t^\eps,0,0,A_t^\eps,\P_{(X_t^\eps,0,0,A_t^\eps)})\big|^2 dt \bigg].
\end{equation}
We also have the following standard estimate for $X^\eps$
\begin{equation}\label{EstimateProof3}
\|X^\eps\|_{\S^2}^2 \ \leq \ C\,\E\bigg[|\xi|^2 + \int_0^T \big|b(t,0,\delta_0,A_t^\eps,\P_{A_t^\eps})\big|^2 dt + \int_0^T \big|\sigma(t,0,\delta_0,A_t^\eps,\P_{A_t^\eps})\big|^2 dt\bigg].
\end{equation}
Now, proceeding as in proof of the stability result \cite[Lemma A.1]{RSZ20_path_reg}, that is applying It\^o's formula to $\langle Y_t^\eps,X_t^\eps\rangle+\langle\bar Y_t^\eps,A_t^\eps\rangle$ between $t=0$ and $t=T$, using the monotonicity conditions of Assumption \ref{ass:two_scaleFB}-4), we obtain
\begin{align*}
&\E\big[\big\langle \phi(0,\delta_0),X_T^\eps\big\rangle\big] \ \leq \ \E\big[\big\langle Y_0^\eps,\xi\big\rangle\big] + \E\big[\big\langle\bar Y_0^\eps,\eta\big\rangle\big] - \lambda\,\E\bigg[\int_0^T\big|A_t^\eps\big|^2 dt\bigg] -\frac{1}{\eps}\E\bigg[\int_0^T\big|\bar Y_t^\eps\big|^2 dt\bigg] \\
&+ \frac{\beta^\eps}{\eps}\E\bigg[\int_0^T\text{tr}\big(\bar Z_t^\eps\Gamma\trans\big)dt\bigg] + \E\bigg[\int_0^T\big\langle\bar Y_t^\eps,b(t,0,\delta_0,0,\delta_0)\big\rangle dt\bigg] + \E\bigg[\int_0^T\textup{tr}\big(\sigma(t,0,\delta_0,0,\delta_0)(Z_t^\eps)\trans\big) dt\bigg] \notag \\
&- \E\bigg[\int_0^T\big\langle\bar X_t^\eps,\Phi(t,0,0,0,0,\delta_0)\big\rangle dt\bigg] 
- \E\bigg[\int_0^T\big\langle A_t^\eps,\Psi(t,0,0,0,0,\delta_0)\big\rangle dt\bigg]. \notag
\end{align*}
This provides in particular the following estimate:
\begin{align*}
&\lambda\,\E\bigg[\int_0^T\big|A_t^\eps\big|^2 dt\bigg] \ \leq \ - \E\big[\big\langle \phi(0,\delta_0),X_T^\eps\big\rangle\big] + \E\big[\big\langle Y_0^\eps,\xi\big\rangle\big] + \E\big[\big\langle\bar Y_0^\eps,\eta\big\rangle\big] \\
&+ \frac{\beta^\eps}{\eps}\E\bigg[\int_0^T\text{tr}\big(\bar Z_t^\eps\Gamma\trans\big)dt\bigg] + \E\bigg[\int_0^T\big\langle\bar Y_t^\eps,b(t,0,\delta_0,0,\delta_0)\big\rangle dt\bigg] + \E\bigg[\int_0^T\textup{tr}\big(\sigma(t,0,\delta_0,0,\delta_0)(Z_t^\eps)\trans\big) dt\bigg] \notag \\
&- \E\bigg[\int_0^T\big\langle\bar X_t^\eps,\Phi(t,0,0,0,0,\delta_0)\big\rangle dt\bigg] 
- \E\bigg[\int_0^T\big\langle A_t^\eps,\Psi(t,0,0,0,0,\delta_0)\big\rangle dt\bigg]. \notag
\end{align*}
Then, it is easy to show that
\begin{align*}
\|A^\eps\|_{\H^2}^2 \ &\leq \ C\bigg(\|X^\eps\|_{\S^2} + \|Y^\eps\|_{\S^2} + \|Z^\eps\|_{\H^2} + \|\bar Y^\eps\|_{\S^2} + \|\bar Z^\eps\|_{\H^2} + \E|\xi|^2 + \E|\eta|^2 + |\phi(0,\delta_0)|^2 \\
&\quad \ + \int_0^T\big|b(t,0,\delta_0,0,\delta_0)\big|^2 dt + \int_0^T\big|\sigma(t,0,\delta_0,0,\delta_0)\big|^2 dt + \int_0^T\big|\Phi(t,0,0,0,0,\delta_0)\big|^2 dt \\
&\quad \ + \int_0^T\big|\Psi(t,0,0,0,0,\delta_0)\big|^2 dt + \frac{(\beta^\eps)^2}{\eps^2}\textup{tr}\big(\Gamma\Gamma\trans\big)\bigg).
\end{align*}
Plugging the latter estimate into \eqref{EstimateProof1}-\eqref{EstimateProof2}-\eqref{EstimateProof3} we obtain estimate \eqref{Estimate_FB}.
\end{proof}

\begin{Assumption}{\bf(D)}\label{ass:optimal_controlFB}
Let Assumption \ref{ass:two_scaleFB} hold. For every $\xi\in L^2(\Omega,\Gc,\P;\R^d)$ there exists $\widehat\alpha\in\H_k^2$ such that
\begin{align}\label{widehat_alpha_FB}
&\big\|\Psi\big(\cdot,\widehat X_\cdot,\widehat Y_\cdot,\widehat Z_\cdot,\widehat\alpha_\cdot,\P_{(\widehat X_\cdot,\widehat Y_\cdot,\widehat Z_\cdot,\widehat\alpha_\cdot)}\big)\big\|_{\H^2}^2 \notag \\
&= \ \E\bigg[\int_0^T \big|\Psi\big(t,\widehat X_t,\widehat Y_t,\widehat Z_t,\widehat\alpha_t,\P_{(\widehat X_t,\widehat Y_t,\widehat Z_t,\widehat\alpha_t)}\big)\big|^2\,dt\bigg] \ = \ 0,
\end{align}
where $(\widehat X,\widehat Y,\widehat Z)\in\S_d^2\times\S_d^2\times\H_{d\times m}^2$ is the unique solution to the following (uncoupled) forward-backward system of stochastic differential equations on $[0,T]$:
\begin{equation}\label{widehat_X_FB}
	\left\{\begin{aligned}d\widehat X_t&=b\big(t,\widehat X_t,\P_{\widehat X_t},\widehat\alpha_t,\P_{\widehat\alpha_t}\big)dt+\sigma\big(t,\widehat X_t,\P_{\widehat X_t},\widehat\alpha_t,\P_{\widehat\alpha_t}\big) dW_t,\\
	d\widehat Y_t&=-\Phi\big(t,\widehat X_t,\widehat Y_t,\widehat Z_t,\widehat\alpha_t,\P_{(\widehat X_t,\widehat Y_t,\widehat Z_t,\widehat\alpha_t)}\big)dt + \widehat Z_t dW_t, \\
	\widehat X_0&=\xi,\ \widehat Y_T=\phi(\widehat X_T,\P_{\widehat X_T}).
	\end{aligned}\right.
\end{equation}
\end{Assumption}

\noindent We can now state the stochastic Tikhonov theorem for the forward-backward system \eqref{TwoScaleFB}.

\begin{theorem}[Stochastic Tikhonov Theorem II]\label{T:ConvergenceII}
Let Assumptions \ref{ass:two_scale}, \ref{ass:two_scaleFB}, \ref{ass:optimal_controlFB} hold. Moreover, let $(\widehat X,\widehat Y,\widehat Z,\widehat\alpha)$ in $\S_d^2\times\S_d^2\times\H_{d\times m}^2\times\H_k^2$ be as in Assumption \ref{ass:optimal_controlFB} and, for every $\varepsilon>0$, let $(X^\varepsilon,Y^\eps,Z^\eps,\bar Y^\eps,\bar Z^\eps,A^\varepsilon)$ in $\S_d^2\times\S_d^2\times\H_{d\times m}^2\times\S_k^2\times\H_{k\times m}^2\times\S_k^2$ be as in Proposition \ref{P:System_eps_FB}. Then, it holds that
\begin{equation*}
\big\|(A^\eps,Y^\eps,Z^\eps,\bar Y^\eps,\bar Z^\eps)-(\widehat\alpha,\widehat Y,\widehat Z,0,0)\big\|_{\H^2} \ \underset{\varepsilon\rightarrow0}{\longrightarrow} \ 0, \qquad\qquad \big\|X^\varepsilon - \widehat X\big\|_{\S^2} \ \underset{\varepsilon\rightarrow0}{\longrightarrow} \ 0.
\end{equation*}
\end{theorem}
\begin{proof}
Let $\widehat\alpha\in\H_k^2$ be the process appearing in Assumption \ref{ass:optimal_controlFB}. By Lemma \ref{L:Density} there exists $(F^n)_n\in\H_k^2$ such that $\|\alpha^n-\widehat\alpha\|_{\H^2}\rightarrow0$ as $n\rightarrow\infty$, with $\alpha_t^n=\int_0^t F_s^n ds$, $t\in[0,T]$ (as in \eqref{tilde_alpha}, we redefine $\alpha^n$ at $t=0$ as: $\alpha_0^n=\eta$). Then, for every $n\in\N$, let $(X^n,Y^n,Z^n)\in\S_d^2\times\S_d^2\times\H_{d\times m}^2$ be the unique solution to the following (uncoupled) forward-backward system of stochastic differential equations on $[0,T]$:
\begin{equation*}
	\left\{\begin{aligned}dX^n_t&=b\big(t,X^n_t,\P_{X^n_t},\alpha^n_t,\P_{\alpha^n_t}\big)dt+\sigma\big(t,X^n_t,\P_{X^n_t},\alpha^n_t,\P_{\alpha^n_t}\big) dW_t,\\
	dY^n_t&=-\Phi\big(t,X^n_t,Y^n_t,Z^n_t,\alpha^n_t,\P_{(X^n_t,Y^n_t,Z^n_t,\alpha^n_t)}\big)dt + Z^n_t dW_t, \\
	X^n_0&=\xi,\ Y^n_T=\phi(X^n_T,\P_{X^n_T}).
	\end{aligned}\right.
\end{equation*}
Let also $(\bar Y^n,\bar Z^n)\in\S_k^2\times\H_{k\times m}^2$ be the unique solution to the following backward stochastic differential equation on $[0,T]$:
\begin{equation}\label{widehat_X_FB_n_barYZ}
d\bar Y^n_t \ = \ -\Psi\big(t,X^n_t,Y^n_t,Z^n_t,\alpha^n_t,\P_{(X^n_t,Y^n_t,Z^n_t,\alpha^n_t)}\big)dt + \bar Z^n_t dW_t, \qquad \bar Y_T^n \ = \ 0.
\end{equation}
Since $\|\alpha^n-\widehat\alpha\|_{\H^2}\rightarrow0$ as $n\rightarrow\infty$, from standard estimates for (uncoupled) forward-backward systems it holds that
\begin{equation}\label{Convergences_FB_n}
\big\|X^n - \widehat X\big\|_{\S^2} \ \underset{n\rightarrow\infty}{\longrightarrow} \ 0, \qquad \big\|Y^n - \widehat Y\big\|_{\S^2} \ \underset{n\rightarrow\infty}{\longrightarrow} \ 0, \qquad \big\|Z^n - \widehat Z\big\|_{\H^2} \ \underset{n\rightarrow\infty}{\longrightarrow} \ 0.
\end{equation}
Then, recalling \eqref{widehat_alpha_FB}, from \eqref{widehat_X_FB_n_barYZ} we obtain
\begin{equation}\label{Convergences_FB_n_barYZ}
\big\|\bar Y^n\big\|_{\S^2} \ \underset{n\rightarrow\infty}{\longrightarrow} \ 0, \qquad \big\|\bar Z^n\big\|_{\H^2} \ \underset{n\rightarrow\infty}{\longrightarrow} \ 0.
\end{equation}
Now, proceeding as in proof of the stability result \cite[Lemma A.1]{RSZ20_path_reg}, that is applying It\^o's formula to $\langle Y_t^\eps-Y_t^n,X_t^\eps-X_t^n\rangle+\langle\bar Y_t^\eps-\bar Y_t^n,A_t^\eps-\alpha_t^n\rangle$ between $t=0$ and $t=T$, using the monotonicity conditions of Assumption \ref{ass:two_scaleFB}-4), we obtain
\begin{align}\label{IneqProof}
0 \ &\leq \ -\lambda\,\E\bigg[\int_0^T\big|A_t^\eps - \alpha_t^n\big|^2 dt\bigg] -\frac{1}{\eps}\E\bigg[\int_0^T\big\langle\bar Y_t^\eps - \bar Y_t^n,\bar Y_t^\eps\big\rangle dt\bigg] \\
&\quad \ + \E\bigg[\int_0^T\big\langle\bar Y_t^\eps - \bar Y_t^n,F_t^n\big\rangle dt\bigg] + \frac{\beta^\eps}{\eps}\E\bigg[\int_0^T\text{tr}\big((\bar Z_t^\eps - \bar Z_t^n)\Gamma\trans\big)dt\bigg]. \notag
\end{align}
Notice that
\begin{align*}
\int_0^T\big\langle\bar Y_t^\eps - \bar Y_t^n,F_t^n\rangle dt \ &= \ \int_0^T\big\langle\alpha_t^n,\Psi\big(t,X^n_t,Y^n_t,Z^n_t,\alpha^n_t,\P_{(X^n_t,Y^n_t,Z^n_t,\alpha^n_t)}\big) \\
&\quad \ - \Psi\big(t,X_t^\eps,Y_t^\eps,Z_t^\eps,A_t^\eps,\P_{(X_t^\eps,Y_t^\eps,Z_t^\eps,A_t^\eps)}\big)\big\rangle dt + \E\big[\big\langle\eta,\bar Y_0^n - \bar Y_0^\eps\big\rangle\big].
\end{align*}
Moreover, we have
\begin{align*}
&\E\big[\big\langle\eta,\bar Y_0^n - \bar Y_0^\eps\big\rangle\big] \ = \ \E\big[\big\langle\eta,\E\big[\bar Y_0^n - \bar Y_0^\eps\big|\Fc_0\big]\big\rangle\big] \ = \\
&= \ \E\bigg[\int_0^T\big\langle\eta,\Psi\big(t,X^n_t,Y^n_t,Z^n_t,\alpha^n_t,\P_{(X^n_t,Y^n_t,Z^n_t,\alpha^n_t)}\big) - \Psi\big(t,X_t^\eps,Y_t^\eps,Z_t^\eps,A_t^\eps,\P_{(X_t^\eps,Y_t^\eps,Z_t^\eps,A_t^\eps)}\big)\big\rangle\bigg].
\end{align*}
Then, plugging the above equalities into \eqref{IneqProof}, letting $n\rightarrow\infty$, and using convergences \eqref{Convergences_FB_n}-\eqref{Convergences_FB_n_barYZ}, we find
\begin{align}\label{InequalityProof4}
0 \ &\leq \ -\lambda\,\E\bigg[\int_0^T\big|A_t^\eps - \widehat\alpha_t\big|^2 dt\bigg] - \frac{1}{\eps}\E\bigg[\int_0^T\big|\bar Y_t^\eps\big|^2 dt\bigg] \\
&\quad \ - \E\bigg[\int_0^T\big\langle\eta+\widehat\alpha_t,\Psi\big(t,X_t^\eps,Y_t^\eps,Z_t^\eps,A_t^\eps,\P_{(X_t^\eps,Y_t^\eps,Z_t^\eps,A_t^\eps)}\big)\big\rangle dt\bigg] + \frac{\beta^\eps}{\eps}\E\bigg[\int_0^T\text{tr}\big(\bar Z_t^\eps\Gamma\trans\big)dt\bigg]. \notag
\end{align}
From the latter inequality we get
\begin{align*}
\frac{1}{\eps}\E\bigg[\int_0^T\big|\bar Y_t^\eps\big|^2 dt\bigg] \ &\leq \ - \, \E\bigg[\int_0^T\big\langle\eta+\widehat\alpha_t,\Psi\big(t,X_t^\eps,Y_t^\eps,Z_t^\eps,A_t^\eps,\P_{(X_t^\eps,Y_t^\eps,Z_t^\eps,A_t^\eps)}\big)\big\rangle dt\bigg] \\
&\quad \ + \frac{\beta^\eps}{\eps}\E\bigg[\int_0^T\text{tr}\big(\bar Z_t^\eps\Gamma\trans\big)dt\bigg].
\end{align*}
Then, using the linear growth condition of $\Psi$ and estimate \eqref{Estimate_FB}, we deduce that $\|\bar Y^\eps\|_{\H^2}\rightarrow0$ as $\eps\rightarrow0$. Recalling the equation satisfied by $\bar Y^\eps$, this in turn allows to prove that $\|\bar Z^\eps\|_{\H^2}\rightarrow0$ and therefore we get $\|\Psi(t,X_t^\eps,Y_t^\eps,Z_t^\eps,A_t^\eps,\P_{(X_t^\eps,Y_t^\eps,Z_t^\eps,A_t^\eps)})\|_{\H^2}\rightarrow0$ as $\eps\rightarrow0$. Now, using again \eqref{InequalityProof4}, we get
\begin{align*}
\lambda\,\E\bigg[\int_0^T\big|A_t^\eps - \widehat\alpha_t\big|^2 dt\bigg] \ &\leq \ - \, \E\bigg[\int_0^T\big\langle\eta+\widehat\alpha_t,\Psi\big(t,X_t^\eps,Y_t^\eps,Z_t^\eps,A_t^\eps,\P_{(X_t^\eps,Y_t^\eps,Z_t^\eps,A_t^\eps)}\big)\big\rangle dt\bigg] \\
&\quad \ + \frac{\beta^\eps}{\eps}\E\bigg[\int_0^T\text{tr}\big(\bar Z_t^\eps\Gamma\trans\big)dt\bigg].
\end{align*}
This implies that $\|A^\eps-\widehat\alpha\|_{\H^2}\rightarrow0$ as $\eps\rightarrow0$. From \eqref{EstimateX-X'} we then get that $\|X^\eps-\widehat X\|_{\S^2}\rightarrow0$ as $\eps\rightarrow0$. Finally, by standard estimates for backward stochastic differential equations, we conclude that $\|Y^\eps-\widehat Y\|_{\H^2}\rightarrow0$ and $\|Z^\eps-\widehat Z\|_{\H^2}\rightarrow0$ as $\eps\rightarrow0$.
\end{proof}

\section{Approximating control problems with two-scale FBSDEs}
\label{sec:control}

We consider the same probabilistic framework as in Section \ref{sec:Tik}, with the complete probability space $(\Omega,\Fc,\P)$, the $m$-dimensional Brownian motion $W=(W_t)_{t\geq0}$, the sub-$\sigma$-algebra $\Gc$, and the filtration $\F=(\Fc_t)_{t\geq0}$ given by \eqref{Filtr}. We work under the conditions for applying the Pontryagin maximum principle, see for instance \cite{Acciaio_SMP,bookMFG}. 

We now formulate the mean field stochastic optimal control problem. Let $A$ be a closed and convex subset of $\R^k$, which denotes the space of control actions. Let also $\Ac$ be the family of control processes, that is the set of all $\alpha\in\H_k^2$ such that $\alpha$ takes values in $A$.
Given $\xi\in L^2(\Omega,\Gc,\P;\R^d)$ and $\alpha\in\Ac$, the state equation of the mean field control problem reads as follows:
\begin{equation}\label{eq:state}
\begin{cases}
	dX_t=b\big(t,X_t,\P_{X_t},\alpha_t,\P_{\alpha_t}\big)dt+\sigma\big(t,X_t,\P_{X_t},\alpha_t,\P_{\alpha_t}\big) dW_t, \qquad t\in[0,T],\\
\hspace{1.35mm}X_0=\xi.
\end{cases}
\end{equation}
The cost functional to be minimized is defined as
\begin{equation}\label{eq:cost}
	J(\alpha) \ = \ \E\bigg[\int_0^Tf(t,X_t,\P_{X_t},\alpha_t,\P_{\alpha_t})dt+g(X_T,\P_{X_T})\bigg],
\end{equation}
where $X$ solves the controlled equation \eqref{eq:state}. The Hamiltonian function is given by
\begin{equation}\label{eq:Hamiltonian_red}
	H(t,x,\mu,y,z,a,\nu) \ = \ \langle b(t,x,\mu,a,\nu),y\rangle + \textup{tr}\big(\sigma(t,x,\mu,a,\nu)\,z\trans\big) + f(t,x,\mu,a,\nu),
\end{equation}
for every $(t,x,\mu,y,z,a,\nu)\in[0,T]\times\R^d\times\Pc_2(\R^d)\times\R^d\times\R^{d\times m}\times\R^k\times\Pc_2(\R^k)$.

\begin{Assumption}{\bf(E)}[Regularity of the coefficients]\label{ass:smooth}\quad
	\begin{itemize}
		\item The functions $b$, $\sigma$, $f$ are differentiable with respect to $x$ and $a$. Similarly, the function $g$ is differentiable with respect to $x$. Moreover, the function $(x,\mu,a,\nu)\mapsto\partial_x b(t,x,\mu,a,\nu)$ is Lipschitz and has at most linear growth uniformly with respect to $t$. The same conditions apply to $\partial_a b$, $\partial_x\sigma$, $\partial_a\sigma$, $\partial_x f$, $\partial_a f$, $\partial_x g$.
		\item The functions $b$, $\sigma$, $f$ are $L$-differentiable with respect to $\mu$ and $\nu$ (for the definition of $L$-differentiability see Appendix \ref{S:App}). Similarly, the function $g$ is $L$-differentiable with respect to $\mu$. Moreover, the map $(x,\mu,a,\nu,x')\mapsto\partial_\mu b(t,x,\mu,a,\nu)(x')$ is Lipschitz and has at most linear growth uniformly with respect to $t$. The same conditions apply to $\partial_\nu b$, $\partial_\mu\sigma$, $\partial_\nu\sigma$, $\partial_\mu f$, $\partial_\nu f$, $\partial_\mu g$.
\item The functions $\partial_x b$, $\partial_a b$, $\partial_x\sigma$, $\partial_a\sigma$ are bounded. Moreover, the functions $x'\mapsto\partial_\mu b(t,x,\mu,a,\nu)(x')$, $x'\mapsto\partial_\mu\sigma(t,x,\mu,a,\nu)(x')$ (resp. $a'\mapsto\partial_\nu b(t,x,\mu,a,\nu)(a')$, $a'\mapsto\partial_\nu\sigma(t,x,\mu,a,\nu)(a')$) have an $L^2(\mu)$-norm (resp. $L^2(\nu)$-norm) which is uniformly bounded with respect to the other variables.
	\end{itemize}
\end{Assumption}
\begin{Assumption}{\bf(F)}[Convexity]\label{ass:conv} \quad
	\begin{itemize}
		\item The function $g$ is convex: for every $(x',\mu'),(x,\mu)\in\R^d\times\Pc_2(\R^d)$,
		\[	g(x',\mu') \ \geq \ g(x,\mu)+ \langle\partial_x g(x,\mu),x'-x\rangle+\E[\langle\partial_\mu g(x,\mu)(X),X'-X\rangle],
		\]
		where $\P_{X'}=\mu'$, $\P_X=\mu$ (for the definition of $\partial_\mu g$ see Appendix \ref{S:App}).
		\item The function $H$ in \eqref{eq:Hamiltonian_red} is convex in $(x,\mu)$ and strongly convex in $(a,\nu)$: there exist $\lambda_1,\lambda_2>0$ such that, for every $(t,y,z)\in[0,T]\times\R^d\times\R^{d\times m}$, $(x,\mu,a,\nu),(x',\mu',a',\nu')\in\R^d\times\Pc_2(\R^d)\times\R^k\times\Pc_2(\R^k)$,
		\begin{align}
			H(t,x',\mu',y,z,a',\nu')&\ge H(t,x,\mu,y,z,a,\nu) \label{H:convex} \\
			&+\langle\partial_x H(t,x,\mu,y,z,a,\nu),x'-x\rangle+\langle\partial_a H(t,x,\mu,y,z,a,\nu),a'-a\rangle \notag \\
			&+\E[\langle\partial_\mu H(t,x,\mu,y,z,a,\nu)(X),X'-X\rangle+\langle\partial_\nu H(t,x,\mu,y,z,a,\nu)(\alpha),\alpha'-\alpha\rangle] \notag \\
			& +\lambda_1 |a'-a|^2+\lambda_2 \E\big[|\alpha'-\alpha|^2\big], \notag
		\end{align}
	where $\P_{X'}=\mu'$, $\P_X=\mu$, $\P_{\alpha'}=\nu'$, $\P_\alpha=\nu$.
	\end{itemize}
\end{Assumption}
\noindent For any control process $\alpha\in\Ac$, with controlled state $X$ solution to \eqref{eq:state}, we consider the so-called adjoint equation (for the definition of $\partial_\mu H$ see Appendix \ref{S:App}):
\begin{equation*}
\begin{cases}
		dY_t=-\partial_xH(t,X_t,\P_{X_t},Y_t,Z_t,\alpha_t,\P_{\alpha_t})dt-\tilde{\E}\big[\partial_\mu H(t,\tilde{X}_t,\P_{X_t},\tilde{Y}_t,\tilde Z_t,\tilde{\alpha}_t,\P_{\alpha_t})(X_t)\big]dt\\
\hspace{9.5mm}+Z_tdW_t, \hspace{7.8cm} t\in[0,T], \\
		\hspace{1.2mm}Y_T=\partial_xg(X_T,\P_{X_T})+\tilde{\E}\big[\partial_\mu g(\tilde{X}_T,\P_{X_T})(X_T)\big],	
\end{cases}
\end{equation*}
where the triple $(\tilde{X},\tilde{Y},\tilde Z,\tilde{\alpha})$ has the same law as $(X,Y,Z,\alpha)$ and is defined on another probability space $(\tilde{\Omega},\tilde{\cF},\tilde{\P})$, where the expectation is denoted by $\tilde\E$.

\begin{Assumption}{\bf(G)}\label{ass:standing_SMP}
Let Assumptions \ref{ass:smooth}-\ref{ass:conv} hold. For every $\xi\in L^2(\Omega,\Gc,\P;\R^d)$ there exists $\widehat\alpha\in\H_k^2$ satisfying
\begin{equation}\label{eq:critical}
\partial_a H(t,\widehat X_t,\P_{\widehat X_t},\widehat Y_t,\widehat Z_t,\widehat\alpha_t,\P_{\widehat\alpha_t})+\tilde{\E}\big[\partial_\nu H(t,\tilde{X}_t,\P_{\widehat X_t},\tilde{Y}_t,\tilde Z_t,\tilde{\alpha}_t,\P_{\widehat\alpha_t})(\widehat\alpha_t)\big] = 0, \quad dt\otimes d\P\text{-}a.s.
\end{equation}
where $(\tilde{X},\tilde{Y},\tilde Z,\tilde{\alpha})$ has the same law as $(\widehat X,\widehat Y,\widehat Z,\widehat\alpha)$ and is defined on another probability space $(\tilde{\Omega},\tilde{\cF},\tilde{\P})$, with expectation denoted by $\tilde\E$. Moreover, $(\widehat X,\widehat Y,\widehat Z)\in\S_d^2\times\S_d^2\times\H_{d\times m}^2$ is the unique solution to the following (uncoupled) forward-backward system of stochastic differential equations on $[0,T]$:
\begin{equation}\label{eq:SMP_FBSDE}
	\left\{\begin{aligned}d\widehat X_t&=b\big(t,\widehat X_t,\P_{\widehat X_t},\widehat\alpha_t,\P_{\widehat\alpha_t}\big)dt+\sigma\big(t,\widehat X_t,\P_{\widehat X_t},\widehat\alpha_t,\P_{\widehat\alpha_t}\big) dW_t,\\
	d\widehat Y_t&=-\,\partial_xH(t,\widehat X_t,\P_{\widehat X_t},\widehat Y_t,\widehat Z_t,\widehat\alpha_t,\P_{\widehat\alpha_t})dt-\tilde{\E}\big[\partial_\mu H(t,\tilde{X}_t,\P_{\widehat X_t},\tilde{Y}_t,\tilde Z_t,\tilde{\alpha}_t,\P_{\widehat\alpha_t})(\widehat X_t)\big]dt \\
&\,\,\,\,\,\,\,+ \widehat Z_t dW_t, \\
	\widehat X_0&=\xi,\ \widehat Y_T=\phi(\widehat X_T,\P_{\widehat X_T}).
	\end{aligned}\right.
\end{equation}
\end{Assumption}


\begin{theorem}[Stochastic maximum principle]\label{T:SMP}
Under Assumption \ref{ass:standing_SMP}, the process $\widehat\alpha$, satisfying \eqref{eq:critical}, is the unique optimal control process of the mean field control problem \eqref{eq:state}--\eqref{eq:cost}. 
\end{theorem}
\begin{proof}
The strong convexity assumption \eqref{H:convex} on $H$ and the fact that $A$ is a closed and convex subset of $\R^k$ guarantee that $\widehat\alpha$ is the unique minimizer. Then, from \cite[Theorem 3.5]{Acciaio_SMP} the result follows. 
\end{proof}
 
\noindent Theorem \ref{T:SMP} states that, in order to solve the mean field control problem \eqref{eq:state}--\eqref{eq:cost}, it is enough to solve the McKean-Vlasov forward-backward system \eqref{eq:SMP_FBSDE}. Our aim is now to exploit the results of Section \ref{sec:Tik} in order to provide an approximate solution to \eqref{eq:SMP_FBSDE} through a suitable two-scale system. In particular, for every $\eps>0$, $\xi\in L^2(\Omega,\Gc,\P;\R^d)$, $\eta\in L^2(\Omega,\Gc,\P;\R^k)$, consider the following extended two-scale system of stochastic differential equations on $[0,T]$:
\begin{equation}\label{eq:two scales_FBSDE}
\hspace{-1mm}\begin{cases}
		dX^\varepsilon_t=b(t,X^\varepsilon_t,\P_{X^\varepsilon_t},A^\varepsilon_t,\P_{A^\varepsilon_t})dt+\sigma(t,X^\varepsilon_t,\P_{X^\varepsilon_t},A^\varepsilon_t,\P_{A^\varepsilon_t}) dW_t,\\
		dY^\varepsilon_t=-\partial_xH(t,X^\varepsilon_t,\P_{X^\varepsilon_t},Y^\varepsilon_t,Z_t^\eps,A^\varepsilon_t,\P_{A^\varepsilon_t})dt-\tilde{\E}\big[\partial_\mu H(t,\tilde{X}^\varepsilon_t,\P_{X^\varepsilon_t},\tilde{Y}^\varepsilon_t,\tilde Z_t^\eps,\tilde{A}^\varepsilon_t,\P_{A^\varepsilon_t})(X^\varepsilon_t)\big]dt\\
\hspace{1.02cm}+Z^\varepsilon_tdW_t,\\
d\bar Y^\varepsilon_t=-\partial_aH(t,X^\varepsilon_t,\P_{X^\varepsilon_t},Y^\varepsilon_t,Z_t^\eps,A^\varepsilon_t,\P_{A^\varepsilon_t})dt-\tilde{\E}\big[\partial_\nu H(t,\tilde{X}^\varepsilon_t,\P_{X^\varepsilon_t},\tilde{Y}^\varepsilon_t,\tilde Z_t^\eps,\tilde{A}^\varepsilon_t,\P_{A^\varepsilon_t})(A^\varepsilon_t)\big]dt\\
\hspace{1.02cm}+\bar Z^\varepsilon_tdW_t,\\
		\eps dA^\varepsilon_t=-\bar Y_t^\eps dt+\beta^\eps\Gamma dW_t,\\
		X^\varepsilon_0=\xi,\ Y^\varepsilon_T=\partial_x g(X^\varepsilon_T,\P_{X^\varepsilon_T})+\tilde{\E}\big[\partial_\mu g(\tilde{X}^\varepsilon_T,\P_{X^\varepsilon_T})(X^\varepsilon_T)\big], \ \bar Y_T^\varepsilon=0, \ A_0^\eps=\eta,
\end{cases}
\end{equation}
where the triple $(\tilde{X}^\eps,\tilde{Y}^\eps,\tilde Z^\eps,\tilde{A}^\eps)$ has the same law as $(X^\eps,Y^\eps,Z^\eps,A^\eps)$ and is defined on another probability space $(\tilde{\Omega},\tilde{\cF},\tilde{\P})$, where the expectation is denoted by $\tilde\E$. Moreover, $\beta^\varepsilon=o(\varepsilon)$ as $\eps\rightarrow0$ and $\Gamma$ is a matrix of size $k\times m$.

\begin{remark}\label{R:Form_of_the_FBSDE}
Notice that it would be more natural to consider the following system, for every $\eps>0$, $\xi\in L^2(\Omega,\Gc,\P;\R^d)$, $\eta\in L^2(\Omega,\Gc,\P;\R^k)$:
\begin{equation}\label{eq:two scales_FBSDE_II}
\hspace{-3mm}\begin{cases}
		dX^\varepsilon_t=b(t,X^\varepsilon_t,\P_{X^\varepsilon_t},A^\varepsilon_t,\P_{A^\varepsilon_t})dt+\sigma(t,X^\varepsilon_t,\P_{X^\varepsilon_t},A^\varepsilon_t,\P_{A^\varepsilon_t}) dW_t,\\
		dY^\varepsilon_t=-\partial_xH(t,X^\varepsilon_t,\P_{X^\varepsilon_t},Y^\varepsilon_t,Z_t^\eps,A^\varepsilon_t,\P_{A^\varepsilon_t})dt-\tilde{\E}\big[\partial_\mu H(t,\tilde{X}^\varepsilon_t,\P_{X^\varepsilon_t},\tilde{Y}^\varepsilon_t,\tilde Z_t^\eps,\tilde{A}^\varepsilon_t,\P_{A^\varepsilon_t})(X^\varepsilon_t)\big]dt\\
\hspace{1.02cm}+Z^\varepsilon_tdW_t,\\
		\varepsilon dA^\varepsilon_t=-\partial_aH(t,X^\varepsilon_t,\P_{X^\varepsilon_t},Y^\varepsilon_t,Z_t^\eps,A^\varepsilon_t,\P_{A^\varepsilon_t})dt-\tilde{\E}\big[\partial_\nu H(t,\tilde{X}^\varepsilon_t,\P_{X^\varepsilon_t},\tilde{Y}^\varepsilon_t,\tilde Z_t^\eps,\tilde{A}^\varepsilon_t,\P_{A^\varepsilon_t})(A^\varepsilon_t)\big]dt\\
\hspace{1.14cm}+\beta^\varepsilon \Gamma dW_t,\\
		X^\varepsilon_0=\xi,\ Y^\varepsilon_T=\partial_x g(X^\varepsilon_T,\P_{X^\varepsilon_T})+\tilde{\E}\big[\partial_\mu g(\tilde{X}^\varepsilon_T,\P_{X^\varepsilon_T})(X^\varepsilon_T)\big], \ A_0^\eps=\eta,
\end{cases}
\end{equation}
where the triple $(\tilde{X}^\eps,\tilde{Y}^\eps,\tilde Z^\eps,\tilde{A}^\eps)$ has the same law as $(X^\eps,Y^\eps,Z^\eps,A^\eps)$ and is defined on another probability space $(\tilde{\Omega},\tilde{\cF},\tilde{\P})$, where the expectation is denoted by $\tilde\E$. However, in the literature there is not yet well-posedness and stability results for such a system in a McKean-Vlasov setting. As a matter of fact, results \cite[Corollary 2.4]{RSZ20_path_reg} and \cite[Lemma A.1]{RSZ20_path_reg}, that we exploit below for system \eqref{eq:two scales_FBSDE}, rely on a $G$-monotonicity condition (see \cite{RSZ20_path_reg}) that in its present form does not hold in general for system \eqref{eq:two scales_FBSDE_II}.
\end{remark}


\begin{lemma}\label{lem:existence_uniq}
Suppose that Assumption \ref{ass:standing_SMP} holds. Then, for every $\varepsilon>0$, system \eqref{eq:two scales_FBSDE} admits a unique solution $(X^\eps,Y^\eps,Z^\eps,\bar Y^\eps,\bar Z^\eps,A^\eps)\in\S_d^2\times\S_d^2\times\H_{d\times m}^2\times\S_k^2\times\H_{k\times m}^2\times\S_k^2$.
\end{lemma}
\begin{proof}
The result is a direct consequence of \cite[Corollary 2.4]{RSZ20_path_reg}, we simply check that the assumptions required by \cite[Corollary 2.4]{RSZ20_path_reg} hold true. First of all, notice that the Lipschitz and linear growth conditions are a consequence of Assumption \ref{ass:smooth}. Then, it only remains to check the $G$-monotonicity condition reported in the statement of \cite[Corollary 2.4]{RSZ20_path_reg}. We choose $G\in\R^{(d+k)\times (d+k)}$ as $G=I_{d+k}$, the $(d+k)$-dimensional identity matrix. 
Moreover, we choose $\alpha_1=\beta_1=0$, $\beta_2=1$, and $\phi_2(A_1,\bar Y_1,A_2,\bar Y_2)=2(\lambda_1+\lambda_2)\|A_1-A_2\|_{L^2}^2+\frac{1}{\eps}\|\bar Y_1-\bar Y_2\|_{L^2}^2$, $\forall\,(A_i,\bar Y_i)\in L^2(\R^k)\times L^2(\R^k)$, $i=1,2$. With such choices, checking the $G$-monotonicity condition reduces to estimate the following: for every $t\in[0,T]$, $(X_i,A_i,Y_i,Z_i,\bar Y_i,\bar Z_i)\in L^2(\R^d)\times L^2(\R^k)\times L^2(\R^d)\times L^2(\R^{d\times m})\times L^2(\R^k)\times L^2(\R^{k\times m})$, for $i=1,2$,
\begin{align*}
&\E[\langle b(t,X_1,\P_{X_1}, A_1,\P_{A_1}) - b(t, X_2,\P_{X_2}, A_2,\P_{A_2}),Y_1 - Y_2\rangle ] - \frac{1}{\eps^3}\|\bar Y_1 - \bar Y_2\|_{L^2}^2 \\
&+ \E\big[\text{tr}\big((\sigma(t,X_1,\P_{X_1},A_1,\P_{A_1}) - \sigma(t,X_2,\P_{X_2},A_2,\P_{A_2}))(Z_1-Z_2)\trans\big)\big] \\
&-\E[\langle \partial_x H(t,X_1, \P_{X_1},Y_1,Z_1, A_1, \P_{A_1}) - \partial_x H(t,X_2, \P_{X_2},Y_2,Z_2, A_2, \P_{A_2}), X_1 -X_2\rangle ]\\
&-\E\big[\tilde{\E}[\langle \partial_\mu H(t,X_1, \P_{X_1},Y_1,Z_1, A_1, \P_{A_1})(\tilde{X}_1) - \partial_\mu H(t,X_2, \P_{X_2},Y_2,Z_2, A_2, \P_{A_2})(\tilde{X}_2), \tilde{X}_1 -\tilde{X}_2\rangle] \big] \\
&- \E[\langle\partial_a H(t,X_1,\P_{X_1},Y_1,Z_1,A_1,\P_{A_1}) - \partial_a H(t,X_2,\P_{X_2},Y_2,Z_2,A_2,\P_{A_2}),A_1 - A_2\rangle] \\
&- \E\big[\tilde{\E}[\langle\partial_\nu H(t,X_1,\P_{X_1},Y_1,Z_1,A_1,\P_{A_1})(\tilde A_1) - \partial_\nu H(t,X_2,\P_{X_2},Y_2,Z_2,A_2,\P_{A_2})(\tilde A_1),\tilde A_1-\tilde A_2\rangle]\big],
\end{align*}
where $(\tilde X_1,\tilde A_1,\tilde{\bar Y}_1,\tilde X_2,\tilde A_2,\tilde{\bar Y}_2)$ has the same law as $(X_1,A_1,\bar Y_1,X_2,A_2,\bar Y_2)$ and is defined on another probability space $(\tilde\Omega,\tilde\Fc,\tilde\P)$, where the expectation is denoted by $\tilde\E$. From the convexity assumption \eqref{H:convex} on $H$, we obtain that the above quantity is bounded from above by
\begin{align*}
&-2(\lambda_1+\lambda_2)\|A_1-A_2\|_{L^2}^2 - \frac{1}{\eps^3}\|\bar Y_1 - \bar Y_2\|_{L^2}^2, 
\end{align*}
which shows the validity of the $G$-monotonicity condition.
\end{proof}

\noindent We now compare the two FBSDE systems \eqref{eq:SMP_FBSDE} and \eqref{eq:two scales_FBSDE}. In the first case, if $\widehat{\alpha}$ is not known, it needs to be numerically computed. This means that every discretization of the system \eqref{eq:SMP_FBSDE} involves, \emph{at each step}, a further numerical approximation of the argmin of the Hamiltonian $H$. In \eqref{eq:two scales_FBSDE} this further computation is avoided entirely and it is simply replaced by a forward-backward system of SDEs, namely the third and fourth equations, which plays the role of a stochastic gradient descent method which evolves at a different time scale.

\noindent The following is the main result of this section.
We prove that, as $\varepsilon\to 0$, the solution of the two-scale forward-backward system \eqref{eq:two scales_FBSDE} converges to the solution of \eqref{eq:SMP_FBSDE}.

\begin{theorem}\label{thm:main_SMP}
Suppose that Assumption \ref{ass:standing_SMP} holds. Let $(\widehat X,\widehat Y,\widehat Z)\in\S_d^2\times\S_d^2\times\H_{d\times m}^2$ be the unique solution of \eqref{eq:SMP_FBSDE} and, for every $\varepsilon>0$, let $(X^\eps,Y^\eps,Z^\eps,\bar Y^\eps,\bar Z^\eps,A^\eps)\in\S_d^2\times\S_d^2\times\H_{d\times m}^2\times\S_k^2\times\H_{k\times m}^2\times\S_k^2$ be the unique solution of \eqref{eq:two scales_FBSDE}. Let also $\widehat\alpha$ be the process satisfying \eqref{eq:critical}. Then
\[
\big\|(Y^\varepsilon,Z^\varepsilon,\bar Y^\eps,\bar Z^\eps,A^\eps) - (\hat{Y},\hat{Z},0,0,\widehat\alpha)\big\|_{\H^2} \ \underset{\varepsilon\rightarrow0}{\longrightarrow} \ 0
\]
and
\[
\big\|X^\varepsilon - \widehat X\big\|_{\S^2} \ \underset{\varepsilon\rightarrow0}{\longrightarrow} \ 0.
\]
\end{theorem}
\begin{proof}
The result follows directly from Theorem \ref{T:ConvergenceII}.
\end{proof}

\subsection{A McKean-Vlasov linear quadratic example}
As an illustration of our results, we consider the classical Linear Quadratic (LQ) control problem extended to the McKean-Vlasov setting. After briefly introducing the model, we outline a method to find its solution which is based on the derivation of suitable Riccati equations. The classical approach of solving \eqref{eq:SMP_FBSDE} will be used as a benchmark for evaluating the performance of our novel approach, which instead relies on solving the two-scale system \eqref{eq:two scales_FBSDE}. We elaborate on the following two main points. First, we focus on the goodness of the approximation of the optimal control and the optimally controlled state, as $\varepsilon$ goes to zero. Second, we argue why using the two scale approximation may be preferable to the classical approach when the computation of the argmin is costly, providing evidence that our new method can strongly outperform the standard one.

We now present the model. In the rest of the section, we use the notation $\overline{\mu}:=\int x d\mu$ to denote the first moment of a certain distribution $\mu\in\Pc_2(\R^d)$. We set:

\begin{itemize}
	\item $b(t,x,\mu,a)=b_1(t)x+b_2(t)a+b_3(t)\overline{\mu}$,
	\item $\sigma\in \R^{d\times m}$,
	\item $f(t,x,\mu,a)=\tfrac12 [ x\trans q_1(t)x+ \overline{\mu}\trans q_2(t)\overline{\mu}+a\trans r(t)a]$,
	\item $g(x,\mu)= \tfrac12[ x\trans q_1x + \overline{\mu}\trans q_2 \overline{\mu} ]$,
	\item $A=\R^k$,
\end{itemize}
for some $q_1,q_2\in \R^{d\times d}$ and some bounded measurable functions $b_1(t)$, $b_2(t)$, $b_3(t)$, $q_1(t)$, $q_2(t)$, $r(t)$ taking values respectively in $\R^{d\times d}$, $\R^{d\times k}$, $\R^{d\times d}$, $\R^{d\times d}$, $\R^{d\times d}$, $\R^{d\times d}$, with $q_1(t)$, $q_2(t)$, $q_1$, $q_2$, $r(t)$ being symmetric and such that $q_1(t), q_2(t), q_1, q_2 \geq 0$, $r(t)\geq \lambda_r I_k$, for some $\lambda_r>0$, where $I_k$ is the identity matrix of order $k$. The Hamiltonian of the system reads as
\begin{align*}
	H(t,x,\mu,y,z,a) \ = \ \big\langle b_1(t)x+b_2(t)a +b_3(t)\overline{\mu},y\big\rangle + \text{tr}\big(\sigma z\trans\big) +\frac{1}{2}\big(x\trans q_1(t)x+ \overline{\mu}\trans q_2(t) \overline{\mu} +a\trans r(t)a\big).
\end{align*}
Therefore, solving the first-order condition, we see that the minimizer of $H$ is given by 
\begin{equation}\label{eq:opt_control_LQ1}
	\widehat{\alpha}(t,x)=-r(t)^{-1}b_2(t)\trans y.
\end{equation}
The McKean-Vlasov system \eqref{eq:SMP_FBSDE} reads as follows:
\begin{equation}\label{eq:SMP_LQ1}
	\left\{\begin{aligned}
		dX_t &=\big(b_1(t)X_t-b_2(t)r(t)^{-1}b_2(t)\trans Y_t+b_3(t)\E[X_t]\big) dt+\sigma dW_t,\\
		dY_t & =-\Big( (b_1(t)+b_3(t))\trans Y_t +q_1(t) X_t + q_2(t)\E[X_t])\Big) dt
		+ Z_tdW_t,\\
		X_0&=\xi, \quad  Y_T= q_1  X_T +  q_2 \E[X_T].
	\end{aligned}\right.
\end{equation}

\begin{figure}
	\centering
	\includegraphics[width=0.49\textwidth]{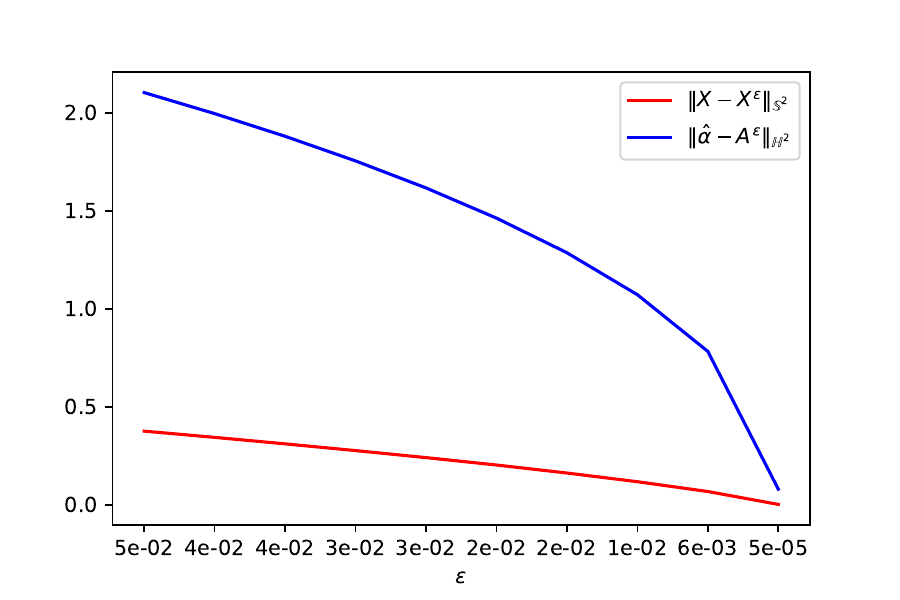}
	\includegraphics[width=0.49\textwidth]{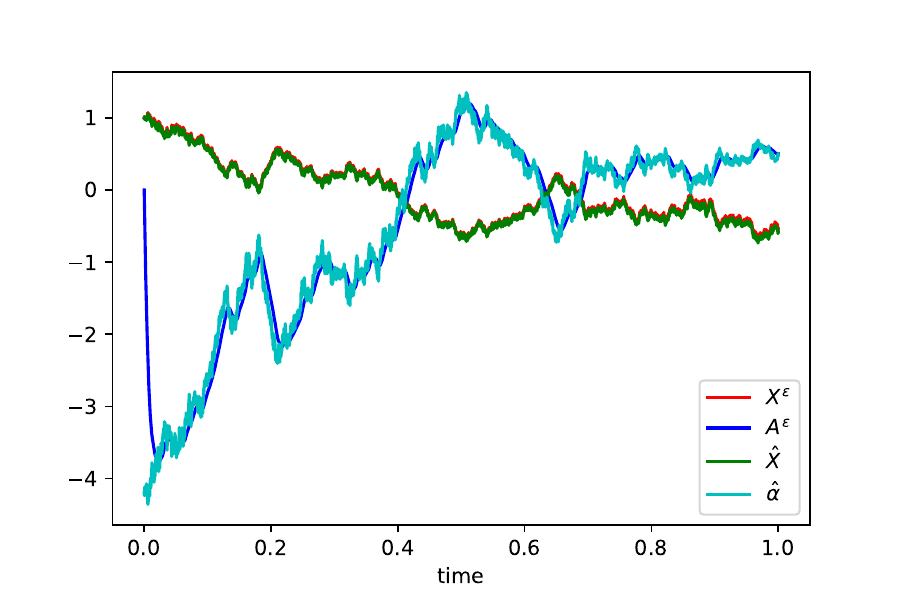}    
	\caption{{\footnotesize Approximation errors with varying $\varepsilon$ (Left). An example of trajectory of $X,\hat{\alpha},X^\varepsilon,A^\varepsilon$ for $\eps=5\cdot10^{-5}$ (Right).} }
	\label{fig:numerics}
\end{figure}

\noindent By taking the expectation in \eqref{eq:SMP_LQ1}, one can find an ODE system for $\E[X_t]$ and $\E[Y_t]$. Given the linearity of the system, one expects its solution to be of the form $\E[Y_t]=\lambda_1(t)\E[X_t]$, for a deterministic function $\lambda_1$. Calculating the differential of this expression and comparing it to the one from the ODE system, one obtains a first Riccati equation: 
\begin{align*}
&\lambda_1^\prime(t)+\lambda_1(t)\big( b_1(t)+b_3(t) - b_2(t)r(t)^{-1}b_2(t)\lambda_1(t)\big) 
+ (b_1(t)+b_3(t))\trans \lambda_1(t) + q_1(t) +q_2(t)=0, \\
&\lambda_1(T) = q_1+q_2.
\end{align*}
When the solution exists, it yields expressions for $\E[X_t]$ and $\E[Y_t]$ that can be plugged into $\eqref{eq:SMP_LQ1}$. This makes the system of FBSDE no longer of McKean-Vlasov type. Again, given the linearity of the system, one expects its solution to be of the form $Y_t=\lambda_2(t)X_t+\theta_2(t)$, for some deterministic functions $\lambda_2$, $\theta_2$. Computing the (stochastic) differential of this expression and comparing to the one of the FBSDE system, one obtains the following Riccati equations: 
\[ 
	\begin{split}
&\lambda_2^\prime(t)+\lambda_2(t)\big( b_1(t) - b_2(t)r(t)^{-1}b_2(t)\lambda_2(t)\big) 
+ (b_1(t)+b_3(t))\trans \lambda_2(t) + q_1(t) =0, \\
& \theta_2^\prime(t) -\lambda_2(t) b_2(t)r(t)^{-1}b_2(t)\theta_2(t) 
+ (b_1(t)+b_3(t))\trans \theta_2(t) + q_2(t) \E[X_t] + \lambda_2(t) b_3(t) \E[X_t] =0, \\
&\lambda_2(T) = q_1, \qquad \theta_2(T)= q_2. 	
\end{split} 
\]
We can then exploit the solutions $\lambda_2$, $\theta_2$ to decouple system \eqref{eq:SMP_LQ1} and solve the resulting SDE for $X$ using a standard Euler-Maruyama scheme.

For the analysis of more general linear-quadratic McKean-Vlasov control problems with controlled volatility and dependence also on the mean of the control, under slightly different assumptions that ours, we refer to \cite{Yong2013}. This section, however, aims at illustrating that the two-scale method performs better that the standard one, which turns out to be true even for the simplest linear-quadratic examples (even non McKean-Vlasov). We thus leave the numerical study of the more complicated case of controlled volatility for future research.

\paragraph{The two-scale approximation.}

The FBSDE system reads as follows:
\begin{equation}\label{eq:two scales_LQ2}
	\left\{\begin{aligned}
		dX^\varepsilon_t=&\big(b_1(t)X^\varepsilon_t+b_2(t)A^\varepsilon_t+b_3(t)\E[X^\varepsilon_t]\big) dt+\sigma dW_t,\\
		dY^\varepsilon_t=&-\Big((b_1(t)+b_3(t))\trans Y^\varepsilon_t +q_1(t) X^\varepsilon_t + q_2(t)\E[X^\varepsilon_t]\Big) dt
		+ Z_t^\varepsilon dW_t,\\
		\varepsilon dA^\varepsilon_t=&- \overline{Y}_t^\varepsilon dt + \beta^\varepsilon \Gamma d W_t , \\
		d \overline{Y}_t^\varepsilon =& -
		\big(b_2(t)\trans Y^\varepsilon_t+r(t)A^\varepsilon_t\big)dt
		+ \overline{Z}_t^\varepsilon dW_t \\
		X_0=&\xi,\quad Y^\varepsilon_T= q_1  X^\varepsilon_T + q_2 \E[X^\varepsilon_T], 
		\quad \overline{Y}^\varepsilon_T = 0, \quad A_0^\varepsilon = \eta.
	\end{aligned}\right.
\end{equation}
To solve this linear system, we proceed as above. We first consider the ODE for the expectations $\E[X_t]$, $\E[Y_t]$, $\E[A_t]$, $\E[\overline{Y}_t]$. Making the ansatz 
$(\E[Y_t], \E[\overline{Y}_t])\trans =\Lambda_1(t) (\E[X_t], \E[A_t])\trans$, for a deterministic matrix $\Lambda_1(t) \in\R^{(d+k)\times (d+k)}$, we derive the ODE
\[
	\begin{split}
		&\Lambda_1^\prime(t)+\Lambda_1(t)\big(C_1(t) + D_1(t)\Lambda_1(t)\big) 
		+ E_1(t)  \Lambda_1(t) + F_1(t)=0 , \\
		& \Lambda_1(T)=  G_1,
	\end{split}
\] 
where
\begin{align*}
	C_1(t)&= \left[ {\begin{array}{cc}
			b_1(t)+b_3(t) & b_2(t) \\
			0 & 0  \\
	\end{array} } \right] ,
\quad D_1(t)= \left[ {\begin{array}{cc}
		0 & 0 \\
		0 & -I_k/\varepsilon  \\
\end{array} } \right] ,
\quad E_1(t)= \left[ {\begin{array}{cc}
		(b_1(t)+b_3(t))\trans & 0 \\
		b_2(t)\trans & 0  \\
\end{array} } \right] ,
\\
F_1(t) &= \left[ {\begin{array}{cc}
		q_1(t)+q_3(t) & 0) \\
		0 & r(t)  \\
\end{array} } \right] ,
\qquad G_1=  \left[ {\begin{array}{cc}
	 q_1+q_2 & 0 \\
		0 & 0  \\
\end{array} } \right] .
\end{align*}
Solving the equation for the mean, we get $\E[X_t]$, so that we can solve the FBSDE \eqref{eq:two scales_LQ2}: making the ansatz $(Y_t, \overline{Y}_t)\trans =\Lambda_2(t) (X_t, A_t)\trans + \Theta_2(t)$, we find the Riccati equations
\[
	\begin{split}
		&\Lambda_2^\prime(t)+\Lambda_2(t)\big(C_2(t) + D_1(t)\Lambda_2(t)\big) 
	+ E_1(t)  \Lambda_1(t) + F_2(t)=0 , \\ 
	& \Theta_2^\prime(t) + \Lambda_2(t)D_1(t) \Theta_2(t) + E_1(t) \Theta_2(t) +\Lambda_2(t) \mu_2(t) + \nu_2(t) =0, \\
	& \Lambda_2(T)=  G_2, \qquad \Theta_2(T) = g_2,
	\end{split}
\] 
where 
\begin{align*}
	C_2(t)&= \left[ {\begin{array}{cc}
			b_1(t) & b_2(t) \\
			0 & 0  \\
	\end{array} } \right] ,
	\quad 
	F_2(t) = \left[ {\begin{array}{cc}
			q_1(t) & 0) \\
			0 & r(t)  \\
	\end{array} } \right] ,
	\qquad G_2=  \left[ {\begin{array}{cc}
			q_1 & 0 \\
			0 & 0  \\
	\end{array} } \right] , \\
 \mu_2(t) &= \left[ {\begin{array}{cc}
		b_3(t)\E[X_t] \\
		0   \\
\end{array} } \right], 
\qquad \nu_2(t) = \left[ {\begin{array}{cc}
		q_2(t)\E[X_t] \\
		0   \\
\end{array} } \right],
\qquad  g_2 = \left[ {\begin{array}{cc}
		q_2\E[X_T] \\
		0   \\
\end{array} } \right].
\end{align*}
We can then exploit the solutions $\Lambda_2$, $\Theta_2$ to decouple system \eqref{eq:two scales_LQ2} and solve the resulting SDEs for $(X^\varepsilon, A^\varepsilon)$ using a standard Euler-Maruyama scheme.

In Figure \ref{fig:numerics}, we present the results of our experiments on a one-dimensional problem, namely $d=k=m=1$, with coefficients $b_1(t)=b_2(t)=b_3(t)=q_1(t)=q_1=r=\Gamma=1$, $q_2(t)=q_2=0.1$, $\beta^\varepsilon= \varepsilon^2$ and initial conditions $\xi=1$, $\eta=0$. We solve \eqref{eq:two scales_LQ2} for different values of $\varepsilon$, varying from $5\cdot10^{-2}$ to $5\cdot10^{-5}$. On the left of Figure \ref{fig:numerics} we plot the $\S^2$-error of $X^\varepsilon$ with respect to the optimally controlled state $\hat{X}$ and the $\H^2$-error of $A^\varepsilon$ with respect to $\hat{\alpha}$ which, in this case, can be computed explicitly and it is given by \eqref{eq:opt_control_LQ1}. We can appreciate how the errors decrease with $\varepsilon$ and they become small, as we expect from Theorem \ref{thm:main_SMP}. On the right of Figure \ref{fig:numerics} we show an example of a trajectory for the quantities $(\hat{X},\hat{\alpha},X^\varepsilon,A^\varepsilon)$ for $\varepsilon=5\cdot10^{-5}$. We observe how both $X^\varepsilon$ and $A^\varepsilon$ are good approximations of $\hat{X}$ and $\widehat{\alpha}$ respectively. We note that $X^\varepsilon$ has the same initial point as $\hat{X}$ and this is a known parameter. In general, how to initialize $A^\varepsilon$ is not a priori clear and an initial condition far from the true optimal control may obviously lead to a larger approximation error.

\paragraph{The curse of dimension.}
In this paragraph we put ourselves in a classical setting (non McKean-Vlasov), which is a particular case of the above general setting obtained by choosing $b_3(t)=q_2(t)=q_2= 0$. For the coefficients $b_1(t),b_2(t),q_1(t),q_1$ we simply take the identity matrix (in particular, we assume that $d=k$). Finally, we choose
\[
r(t)=[r_{ij}],\quad\text{with}\quad r_{ij} = \delta_{ii} (1+\exp(-i))^{-1},
\]
where $\delta$ is the Kronecker delta. The purpose of this example is two-fold. First, we show that even in the framework of this simple example, where also the minimizer of the Hamiltonian has an explicit solution, the classical approach is severely outperformed by the two-scale approximation, as the dimension of the state space grows. Second, we stress the fact that the results of the paper are new even in the non McKean-Vlasov setting.

We start noting that, although explicitly given by \eqref{eq:opt_control_LQ1}, the computation of the minimizer requires several matrix operations including inversion, transpose and matrix multiplication. All of them can be costly when the dimension is large. We experimented the classical approach and the two-scale approach using the specifications above and changing only the dimension $d$. The result are presented in Table \ref{table:dimensions}, where we referred to the approach using the minimization of the Hamiltonian as the \emph{argmin} approach (AM) and compared it to the \emph{two-scale} (TS) approach with $\eps=5\cdot10^{-5}$.
\begin{table}
	\begin{center}
		\begin{tabular}{|c|ccc|ccc|}
			\hline
			Dimension & Exec. time AM & Exec. time TS& Time saved & Exp. cost AM  & Exp. cost TS & Rel. Error\\
			\hline
			10  &      0.78s &      1.10s &    -41.9\% &      18.65 &     19.05 &    2.17\%\\
			20  &      0.95s &      1.22s &   -28.73\% &      38.53 &     39.36 &    2.13\%\\
			30  &      1.50s &      1.38s &     8.17\% &      61.16 &     62.47 &    2.14\%\\
			50  &      6.49s &      1.97s &    69.67\% &     105.70 &    107.99 &    2.17\%\\
			75  &     37.55s &      2.88s &    92.33\% &     158.01 &    161.39 &    2.14\%\\
			100 &     80.38s &      4.34s &    94.60\% &     212.72 &    217.28 &    2.14\%\\
			125 &    153.32s &      8.27s &    94.60\% &     269.95 &    275.74 &    2.14\%\\
			150 &    298.89s &     12.06s &    95.97\% &     322.55 &    329.49 &    2.15\%\\
			200 &    761.94s &     27.58s &    96.38\% &     432.70 &    442.00 &    2.15\%\\	
			\hline
		\end{tabular}
		\caption{{\footnotesize Execution times and expected costs for the argmin (AM) approach and the two-scale (TS) approach. Computed on a Surface 6 pro laptop with Intel(R) Core(TM) i5-8350U CPU 1.70 GHz.}}
		\label{table:dimensions}
	\end{center}
\end{table}
As we can see from the results of the table, in lower dimension it is still more convenient to minimize the Hamiltonian. As soon as the dimension increases, the two-scale approach performs significantly better than the argmin one, with an essentially constant error of approximation. To appreciate visually the data of Table \ref{table:dimensions} we can see, in Figure \ref{fig:ex_times}, the difference of the execution times of the two algorithms, when the dimension increases.
\begin{figure}
	\centering
	\includegraphics[width=0.6\textwidth]{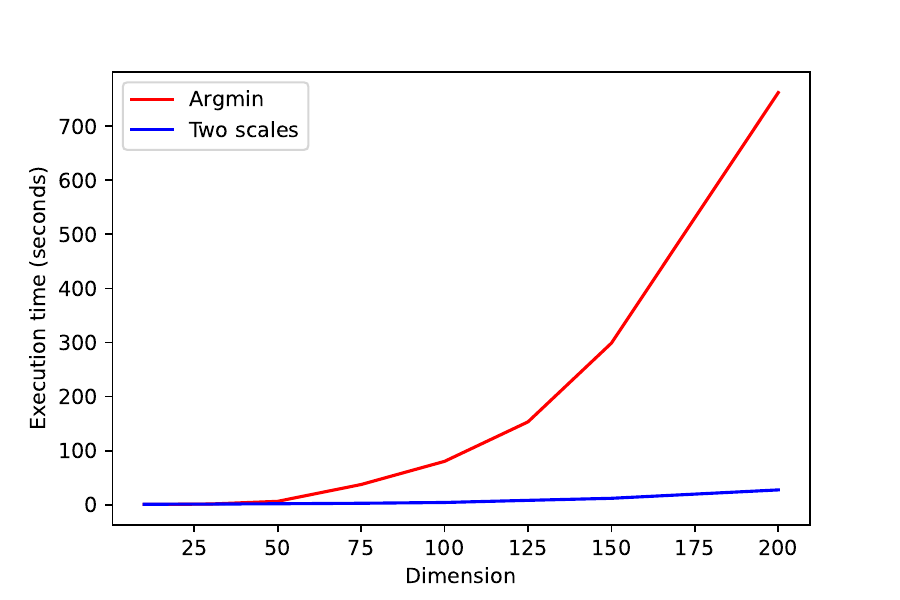} 
	\caption{{\footnotesize 
			Execution times with varying dimension.}}
	\label{fig:ex_times}
\end{figure}
In addition to the data reported in Table \ref{table:dimensions}, it is worth mentioning that with $d=500$ the two-scale algorithm concludes in 2m:18s, whereas the argmin one has an estimated execution time of 7h:57m:45s. From these numerical results we conclude that the two-scale method may represent a valid fast alternative for solving problems in high dimension.

\appendix

\section{Lions differentiability}
\label{S:App}

There are various notions of differentiability for functions of probability measures. In the present paper we adopt the definition firstly given by P.-L. Lions in the series of lectures \cite{Lions}. In the present appendix we present the essential features of such a definition, while we refer to \cite[Chapter 5]{bookMFG} for more details.\\
Lions' definition is based on the idea of \emph{lifting}, which allows to interpret a derivative with respect to a measure as a Fr\'echet derivative in the $L^2$ space of random variables. Firstly, fix a probability space $(\Omega,\Fc,\P)$, not necessarily the same adopted in the previous sections of the paper, satisfying the following property: there exists an $\Fc$-measurable random variable $U\colon\Omega\rightarrow\R$ having uniform distribution on $[0,1]$. We recall that the probability space adopted in the previous sections satisfies such a property and, in particular, $U$ can be taken $\Gc$-measurable. We also recall from Remark \ref{R:Uniform} that the existence of $U$ is equivalent to the following property:
\[
\Pc_2(\R^n) \ = \ \big\{\P_\xi\colon\xi\in L^2(\Omega,\Fc,\P;\R^n)\big\}, \qquad \forall\,n\in\N.
\]
\begin{definition}
Given a function $u\colon\Pc_2(\R^d)\rightarrow\R$, we say that $\tilde u\colon L^2(\Omega,\Fc,\P;\R^d)\rightarrow\R$ is the \textbf{lifting} of $u$ if it holds that
\[
\tilde u(\xi) \ = \ u(\P_\xi), \qquad \forall\,\xi\in L^2(\Omega,\Fc,\P;\R^d).
\]
\end{definition}

\begin{definition}
Given a function $u\colon\Pc_2(\R^d)\rightarrow\R$ and a probability $\mu_0\in\Pc_2(\R^d)$, we say that $u$ is \textbf{differentiable in the sense of Lions} or $L$-\textbf{differentiable} at $\mu_0$ if there exists $\xi_0\in L^2(\Omega,\Fc,\P;\R^d)$ such that $\P_{\xi_0}=\mu_0$ and its lifting $\tilde u$ is differentiable in the sense of Fr\'echet at $\xi_0$.
\end{definition}

\begin{definition}
Let $u\colon\Pc_2(\R^d)\rightarrow\R$ such that its lifting $\tilde u$ is everywhere differentiable in the sense of Fr\'echet. We say that $u$ admits $L$-\textbf{derivative} if there exists a function $\partial_\mu u$ defined on $\Pc_2(\R^d)$, such that $\mu_0\mapsto\partial_\mu u(\mu_0)(\cdot)\in L^2(\R^d,\Bc(\R^d),\mu_0;\R^d)$ and
\[
D\tilde u(\xi_0) \ = \ \partial_\mu u(\mu_0)(\xi_0), \qquad \P\text{-a.s.}
\]
for every $\xi_0\in L^2(\Omega,\Fc,\P;\R^d)$ with $\P_{\xi_0}=\mu_0$. The function $\partial_\mu u$ is called $L$-\textbf{derivative} of $u$.
\end{definition}

\begin{proposition}
Let $u\colon\Pc_2(\R^d)\rightarrow\R$ such that its lifting $\tilde u$ is everywhere differentiable in the sense of Fr\'echet and $D\tilde u\colon L^2(\Omega,\Fc,\P;\R^d)\rightarrow L^2(\Omega,\Fc,\P;\R^d)$ is a continuous function. Then, for every $\mu_0\in\Pc_2(\R^d)$, there exists a Borel-measurable function $g_{\mu_0}\colon\R^d\rightarrow\R^d$ such that
\[
D\tilde u(\xi_0) \ = \ g_{\mu_0}(\xi_0), \qquad \P\text{-a.s.}
\]
for every $\xi_0\in L^2(\Omega,\Fc,\P;\R^d)$ with $\P_{\xi_0}=\mu_0$.
\end{proposition}
\begin{proof}
See \cite[Proposition 5.25]{bookMFG}.
\end{proof}

\begin{definition}
Given a function $u\colon\Pc_2(\R^d)\rightarrow\R$, we say that $u$ is \textbf{continuously} $L$-\textbf{differentiable} if its lifting $\tilde u$ is everywhere differentiable in the sense of Fr\'echet and $D\tilde u\colon L^2(\Omega,\Fc,\P;\R^d)\rightarrow L^2(\Omega,\Fc,\P;\R^d)$ is a continuous function.
\end{definition}

\begin{proposition}
Let $u\colon\Pc_2(\R^d)\rightarrow\R$ be continuously $L$-differentiable. Then, there at most one function $\partial_\mu u\colon\Pc_2(\R^d)\rightarrow\R$ such that:
\begin{enumerate}[\upshape1)]
\item for every $\mu_0\in\mathscr P_2(\R^d)$, the function $x\mapsto\partial_\mu u(\mu_0,x)$ is Borel-measurable;
\item for every $\xi_0\in L^2(\Omega,\Fc,\P;\R^d)$ with $\P_{\xi_0}=\mu_0$, it holds that
\[
D\tilde u(\xi_0) \ = \ \partial_\mu u(\mu_0)(\xi_0), \qquad \P\text{-q.c.}
\]
\item $\partial_\mu u$ is continuous on $\mathscr P_2(\R^d)\times\R^d$.
\end{enumerate}
If such a function exists then we say that $u$ admits \textbf{continuous} $L$-\textbf{derivative}.
\end{proposition}
\begin{proof}
See \cite[Remark 5.82]{bookMFG}.
\end{proof}

\noindent Finally, let $u\colon\Pc_2(\R^d)\rightarrow\R$ be continuously $L$-differentiable and suppose that $u$ admits $L$-derivative $\partial_\mu u\colon\Pc_2(\R^d)\times\R^d\rightarrow\R^d$ continuous. Then, for every fixed $\mu\in\Pc_2(\R^d)$, we consider the derivative of the function $x\mapsto u(\mu)(x)$, which we denote by $\partial_x\partial_\mu u(\mu)(x)$ and is a function from $\Pc_2(\R^d)\times\R^d$ into $\R^{d\times d}$.

\bibliographystyle{abbrv}
\bibliography{MFG_control}

\end{document}